\documentclass{amsart}
\usepackage{amssymb,enumerate}
\usepackage[chatter]{rotating}
\usepackage{epsfig}
\usepackage{fancyvrb}
\newtheorem{theorem}{Theorem}

\newtheorem{lemma}{Lemma}[section]
\newtheorem{rem}[lemma]{Remark}
\newtheorem{cor}[lemma]{Corollary}
\newtheorem{proposition}[lemma]{Proposition}

\newcommand{\Ax}{{\mathcal A}}

\newcommand{\R}{{\mathbb R}}
\newcommand{\N}{{\mathbb N}}
\newcommand{\Sp}{{\mathfrak  S}}

\newcommand{\Miy}{{\rm Miy }}

\newcommand{\A}{\mathcal A}

\newcommand{\aut}{\rm Aut}
\newcommand{\al}{\alpha}
\newcommand{\bt}{\beta}

\newcommand{\rad}{{\rm rad}}
\newcommand{\ad}{{\rm ad}}

\newcommand{\G}{\overline{G}}
\newcommand{\ov}{\overline{v}}
\newcommand{\ghost}{dormant }
\usepackage{xcolor}

\title[]{$(2B, 3A, 5A)$-subalgebras of the Griess algebra with alternating Miyamoto group}

\author{Clara Franchi}
\address{Dipartimento di Matematica e Fisica,
Universit\`a Cattolica del Sacro Cuore,
Via della Garzetta 48,
I-25133 Brescia, Italy}
\email{clara.franchi@unicatt.it}

\author{Mario Mainardis}
\address{Dipartimento di Scienze Matematiche, Informatiche e Fisiche, 
Universit\`a degli Studi di Udine, via delle Scienze 206,
I-33100 Udine, Italy}
\email{mario.mainardis@uniud.it}
\thanks{}

\subjclass[2010]{Primary  20C30, 20C15, 17B69, 05B99}

\date{}

\dedicatory{}

\commby{}

\begin{document}

\begin{abstract}
We use Majorana representations to study the subalgebras of the Griess algebra that have shape $(2B,3A,5A)$ and whose associated Miyamoto groups are isomorphic to $A_n$. We prove that these subalgebras exist only if $n\in \{5,6,8\}$. The case $n=5$ was already treated by Ivanov, Seress, McInroy, and Shpectorov. In case $n=6$ we prove that these algebras are all isomorphic and provide their precise description. In case $n=8$ we prove that these algebras do not arise from standard Majorana representations.
 \end{abstract}
 
 \maketitle

 \section{Introduction}\label{intro}

 This paper is a contribution to a project aimed at classifying, via Majorana representations, the subalgebras of the Griess algebra whose associated Miyamoto groups appear as factors over their centers of  {\it monstralizers} (see~\cite {N98}) of the subgroups $H$ of the Monster group $M$ that are isomorphic to the alternating group $A_5$ (see~\cite[5.11.3]{FMTII}). In~\cite[Table~3]{N98}, Norton listed the conjugacy classes of such subgroups $H$  and, for each such subgroup, the isomorphism type of its monstralizer. There are $8$ such classes,  labelled with a triple $(2X,3Y,5Z)$  (with $(X,Y,Z)$ in $\{A,B\} \times \{A,B,C\}\times\{A,B\}$), whose entries are the conjugacy classes (in Atlas notation~\cite{ATLAS}) of the $2$, $3$, and $5$ elements of $H$ in $M$.  Majorana representations of  monstralizers of an $A_5$-subgroup of type $(2A,3A,5A)$ (which are isomorphic to $A_{12}$) have been  considered in~\cite{Alonso, FIM2, A12}. In this paper we deal with the class of $A_5$-subgroups of type $(2B,3A,5A)$, whose monstralizers are isomorphic to $2.M_{22}.2$. An important inductive step towards this goal is to determine the Majorana representation induced on a submaximal subgroup $\overline G=Z\times G$ in the derived subgroup of such a monstralizer,  where   $Z$ is a group of order $2$ generated by $z$ and $G$ is isomorphic to $A_6$ (see~\cite[p. 39]{ATLAS}).  In this case, the set of Fischer (i.e. $2A$ in Atlas notation)  involutions in $\overline G$ is the {\it diagonal} set $D:=\{zs\: |\, s \mbox{ is a bitransposition in } A_6\}$  while the involutions in $G$ are Conway (i.e. $2B$) involutions.  
 Since the product of two elements of $D$ does not lie in $D$ and, by~\cite[p.258]{FMTII}, central elements are contained in the kernel of a Majorana representation,  by~\cite[Axiom (M8)]{IPSS}, the action of  $\overline G$ on the Griess algebra, induces a Majorana representation of $\overline G/\langle z \rangle$ (i.e. $A_6$) of shape $(2B,3A,5A)$ (see definition in Section~\ref{sec1}).


 More generally, in this paper we consider Majorana representations of the alternating groups $A_n$, $n\geq 5$, with shape $(2B, 3Y, 5A)$, with $Y$ in $\{A,C\}$. The case $n=5$ has been investigated  by Ivanov, Seress, McInroy, and Shpectorov (see~\cite{IS12, Ser, MS}). By~\cite[Table 3]{N98}, all the Majorana representations of $A_5$ are based on embeddings  of $A_5$ or $2\times A_5$  in $M$. 
 %
 Recall from~\cite{FIM2} that a Majorana representation of $A_n$ is {\it standard} if the bitranspositions of $A_n$ are Majorana involutions. In Section~\ref{Sec:start} we prove 

\begin{theorem}\label{A7} The alternating group  $A_n$ has 
\begin{enumerate}
\item [(a)]  a standard Majorana representation of shape $(2B, 3Y, 5A)$, $Y\in \{A,C\}$, if and only if $n\in\{5,6\}$.  
\item[(b)]   a nonstandard Majorana representation of shape $(2B, 3A, 5A)$, $Y\in \{A,C\}$, if and only if $n=8$.
\end{enumerate}
\end{theorem}

The case $n=6$ and shape $(2B, 3A, 5A)$ seems to be particularly difficult. It has been addressed in~\cite{MS} with the use of the expansion algorithm, but the computation of the representation could not be completed within a reasonable time. 
In this paper we use a different approach, building up inductively a representation $V$ of $A_6$ with shape $(2B, 3A, 5A)$ from its proper subalgebras. Moreover, we introduce {\it \ghost $4$-axes} (see the definition in Section~\ref{sec1}) which are likely to be useful also in other situations, e.g. for the still open classification of the Majorana representations of $A_n$ for $8\leq n\leq 11$ (see~\cite[Conjecture 1]{A12}). Dormant $4$-axes are particular idempotents of length $2$ in $V$ which turn to be $4$-axes whenever $V$ is realised as a subalgebra of the Griess algebra. This fact will be discussed in more detail in Section~\ref{sec1}. Accordingly to what happens in the Griess algebra, we assume that \ghost $4$-axes are in one-to-one correspondence with the cyclic subgroups of order $4$ of $A_6$. We call a representation of $A_6$, with shape $(2B, 3Y, 5A)$ ($Y\in \{A,C\}$) satisfying this extra hypothesis a {\it diagonal Majorana representation}.

 \begin{theorem}\label{A6}
The alternating group $A_6$ has a unique diagonal Majorana representation of shape $(2B, 3A, 5A)$. This representation is $2$-closed of dimension $121$ and based on the embedding of $2\times A_6$ into the Monster.
\end{theorem}

While this paper was in preparation, a Majorana algebra of dimension $121$ and shape $(2B, 3A, 5A)$ was constructed in~\cite{Shumba} as a subalgebra of a larger algebra. Theorem~\ref{A6} answers, under the above extra hypothesis, the question posed in~\cite[p.703]{Shumba} about the uniqueness of that algebra.

In Section~\ref{sec1} we give the basic definitions and recall the known results that will be used in the sequel. Theorem~\ref{A7} is proved in Section~\ref{Sec:start}. From Section~\ref{Scalar} onwards, we fix a diagonal Majorana representation $V$ of $A_6$ with shape $(2B, 3A, 5A)$. In Section~\ref{Scalar} we compute the inner products between axes and odd axes in $V$ and determine the dimensions of the corresponding linear spans. In Section~\ref{maxsub} we consider the subalgebras of $V$ corresponding to the maximal subgroups of $A_6$ isomorphic to $S_4$ or $A_5$ and, by examining their intersections, we derive relations between axes, odd axes, and \ghost $4$-axes which are used in Section~\ref{moreinner} to show that all \ghost $4$-axes are contained in the linear span $U$ of $2$-, $3$-, and $4$-axes. In Section~\ref{5assi} we prove that $U$ has codimension $1$ in the linear span $V^\circ$ of all axes and odd axes. Moreover, we show that the $\R[A_6]$-module structure of $V^\circ$ extends naturally to $\R[S_6]$ and determine its decomposition into irreducible $\R[S_6]$-modules (Theorem~\ref{decv}). Finally, in Section~\ref{closure}, we show that the product in $V$ is uniquely determined and $V^\circ=V$.

 
 \section{Preliminary results}\label{sec1}
 
Let $V$ be a real vector space endowed with a (non-associative) commutative product $\cdot$ and a positive definite symmetric bilinear form $(\:\:, \:)_V$ which associates with the algebra product, i.e. for every $u,v,w\in V$
$$
(u\cdot v, w)_V=(u,v\cdot w)_V.
$$
 For $a\in V$ let $\ad_a$ be the adjoint operator $v\mapsto av$ and for $\lambda \in \R$, denote by $V_\lambda^a$ the $\lambda$-eigenspace for $\ad_a$, i.e. $V_\lambda^a:=\{v\in V\:|\: av=\lambda v\}$. A vector $a\in V$ is called a {\it Majorana axis} if 
\begin{enumerate}
\item $a\cdot a=a$ and $(a, a)_V=1$
\item $\ad_a$ is semisimple with spectrum $S:=\{1,0,\tfrac{1}{4}, \tfrac{1}{32}\}$
\item $V_1^a=\{\lambda a\:|\:\lambda \in \R\}$
\item for every $\lambda, \mu\in S$, the eigenspaces $V_\lambda^a$, $V_\mu^a$ satisfy the Monster fusion law in Table~\ref{fusionlaw}, in the sense that 
$$
V_\lambda^a\cdot V_\mu^a\subseteq \bigoplus_{\nu \in \lambda\star \mu} V_\nu^a .
$$
\begin{table}
$$
 \begin{array}{c||c|c|c|c}
   \star& 1    & 0    & \tfrac{1}{4}      & \tfrac{1}{32}\\
   \hline
   \hline
   1     & 1    &\emptyset    & \tfrac{1}{4}      &\tfrac{1}{32}\\
   \hline
   0     & \emptyset    & 0    & \tfrac{1}{4}       &\tfrac{1}{32}\\
   \hline
   \tfrac{1}{4}  & \tfrac{1}{4}  & \tfrac{1}{4}  & 1,0  & \tfrac{1}{32}\\
   \hline
   \tfrac{1}{32}& \tfrac{1}{32}&\tfrac{1}{32}& \tfrac{1}{32}&1,0,\tfrac{1}{4}\\
    \end{array}
    $$
\caption{Monster fusion law}\label{fusionlaw}
\end{table}
\end{enumerate} 
Conditions (1)-(4) yield that, for every Majorana axis $a$, the linear map $\tau_a:V\to V$ such that 
$$
v^{\tau_a}=\begin{cases}
\:\:\:v & \mbox{ if } v\in V_1^a\oplus V_0^a\oplus V_{\tfrac{1}{4}}^a\\
-v & \mbox{ if } v\in V_{\tfrac{1}{32}}^a
\end{cases}
$$
is an involutory algebra automorphism of $V$ which preserves the bilinear form. $\tau_a$ is called {\it Miyamoto involution (or Majorana involution)} associated to $a$ (see~\cite{Miya}). 

$V$ is called  {\it Majorana algebra} if it is generated (as an algebra) by a set $\Ax$ of Majorana axes. Moreover, usually Norton inequality is assumed, i.e. for every $u,v\in V$
$$
(u\cdot v, u\cdot v,)_V\leq (u\cdot u,v\cdot v)_V.
$$
In this paper, however, we will not make use of the Norton inequality.

Let $V$ be a Majorana algebra generated by the set of axes $\Ax$. The {\it automorphisms} of $V$ are the algebra automorphisms that are also isometries of the vector space $V$. In particular, Miyamoto involutions are automorphisms of $V$. Define the {\it Miyamoto group $\Miy(\Ax)$ of $V$ with respect to} $\Ax$, as follows: 
\begin{equation}\label{miy}
 \Miy(\Ax):=\langle \tau_a \:|\: a\in \Ax \rangle .
 \end{equation} 
Since, by~\cite[Lemma 3.5]{KMS} $ \Miy(\Ax)=\Miy(\Ax^{\Miy(\Ax)})$, one usually assumes (as we shall do in this paper)  that $$\Ax=\Ax^{\Miy(\Ax)}.$$
Let $$G:=\Miy(\Ax) \:\:\mbox{ and }\:\:
T:=\{\tau_a\:|\:a\in \Ax\},
$$
the quintet
$$
(V , (\:, )_V, \cdot ,\Ax , G),
$$
is called a {\it Majorana representation of $G$ on $V$ with respect to $T$}. When the setting is clear, we'll simply say that $V$ is a Majorana representation of $G$. The {\it dimension} of the representation is the dimension of $V$.  
 For $n\in \N$, the {\it $n$-closure} of the representation is the subspace of $V$ spanned by all products of length $n$ of elements in $\Ax$. The representation is {\it $n$-closed} if it is equal to its $n$-closure. 
If $H$ is a subgroup of $G$, denote by $V(H)$ the subalgebra of $V$ generated by all the axes $a\in \Ax$ such that $\tau_a\in H$.

The Griess algebra $V^\sharp$ is a Majorana algebra and, if $2A$ denotes the set of $2A$-axes in $V^\sharp$, then $(V^\sharp, (\:, )_{V^\sharp}, \cdot, 2A, M)$ is a Majorana representation of the Monster group $M$ with respect to the set of Fisher involutions (see e.g.~\cite{Iva}). If $G$ is a subgroup of $M$ generated by a set $T$ of Fisher involutions and $\Ax$ is the set of axes corresponding to $T$, then $(V^\sharp, (\:, )_{V^\sharp}, \cdot,\Ax, G)$ is a Majorana representation of $G$ and we say that it is {\it based on an embedding into the Monster}. 

A fundamental result in the Majorana theory is the combination of Norton's classification of subalgebras of the Griess algebra generated by two axes~\cite{N96} and a similar result of Sakuma~\cite{Sakuma} for VOA's, relying on earlier work of Miyamoto, reproved independently in~\cite{IPSS} in the context of Majorana Theory.
\medskip

\noindent{\bf Norton-Sakuma Theorem}~\cite{N96,Sakuma}
{\it Let $V$ be a Majorana algebra, let $a_0$ and $a_1$ be two axes in $V$, $\rho:= \tau_{a_0}\tau_{a_1}$, and $N:=|\rho|$. For $i\in \N$, let  
$a_{2i}=a_0^{\rho^i}$ and $a_{2i+1}=a_1^{\rho^i}.$  
Then $N\leq 6$ and $\langle \langle a_0, a_1\rangle \rangle$ is one of the nine algebras in Table~\ref{NS}.}
\medskip

 \begin{table}
{\tiny 
$$
\begin{array}{|l|l|l|l|}
\hline

\mbox{Type}& \mbox{Basis} &\mbox{Structure constants} & \mbox{Inner products}\\
\hline
1A & \begin{array}{l} a_0 \end{array}&
\begin{array}{l}
 a_0\cdot a_0=a_0 \end{array}
 &
\begin{array}{l}
(a_0,a_0)=1 \end{array}\\
 \hline
2A 
&\begin{array}{l} a_0,\\ a_1,\\ a_\rho \end{array}
&
\begin{array}{l}
 a_0\cdot a_1=\tfrac{1}{2^3}(a_0+a_1-a_\rho),\\
 a_0\cdot a_\rho=\tfrac{1}{2^3}(a_0+a_\rho-a_1)\\
 a_\rho\cdot a_\rho=a_\rho
 \end{array}
&
\begin{array}{l}
 (a_0, a_1)=\tfrac{1}{2^3}\\
 (a_0, a_\rho)=\tfrac{1}{2^3}\\
 (a_1, a_\rho)=\tfrac{1}{2^3}
 \end{array}
 \\
 \hline
2B &\begin{array}{l} a_0,\\ a_1 \end{array}& \begin{array}{l} 
a_0\cdot a_1=0 \end{array}& \begin{array}{l} (a_0, a_1)=0 \end{array}\\
 \hline

 3A & \begin{array}{l} a_{-1},\\ a_0,\\ a_1,\\ u_\rho \end{array}&
 \begin{array}{l}
 a_0\cdot a_1=\tfrac{1}{2^5}(2a_0+2a_1+a_{-1}) -\tfrac{3^3\cdot 5}{2^{11}}u_\rho\\
  a_0\cdot u_\rho=\tfrac{1}{3^2}(2a_0-a_1-a_{-1})+\tfrac{5}{2^5}u_\rho\\
    u_\rho\cdot u_\rho= u_\rho\\
  \end{array}
  &
  \begin{array}{l}
   (a_0, a_1)=\tfrac{13}{2^8},\\ 
   (a_0, u_\rho)=\tfrac{1}{4},\\
   (u_\rho, u_\rho)=\tfrac{2^3}{5}
 \end{array}\\
 \hline
 3C &\begin{array}{l} a_{-1},\\ a_0,\\ a_1 \end{array}&  \begin{array}{l} a_0\cdot a_1=\tfrac{1}{2^6}(a_0+a_1-a_{-1})\end{array} & \begin{array}{l}(a_0,a_1)=\tfrac{1}{2^6}\end{array} \\
 \hline
 4A & \begin{array}{l}a_{-1}, \\a_0, \\a_1,\\ a_2, \\v_\rho \end{array} 
 &  \begin{array}{l}         
 a_0\cdot a_1= \tfrac{1}{2^6}(3a_0+3a_1+a_{-1}+a_2-3v_\rho) \\
  a_0\cdot a_2=0\\
 a_0\cdot v_\rho= \tfrac{1}{2^4}(5a_0-2a_1-2a_{-1}-a_2+3v_\rho)\\
 v_\rho\cdot v_\rho=v_\rho
 \end{array}
 & \begin{array}{l}  
 (a_0,a_1)=\tfrac{1}{2^5}\\
 (a_0,a_2)=0 \\
 (a_0,v_\rho)=\tfrac{3}{2^3}\\
  (v_\rho,v_\rho)=2
           \end{array} \\
  \hline
4B & \begin{array}{l} a_{-1}, \\a_0,\\ a_1,\\a_2\\ a_{\rho^2} \end{array}
&  \begin{array}{l}   
a_0\cdot a_1= \tfrac{1}{2^6}(a_0+a_1-a_{-1}-a_2+a_{\rho^2}) \\
 a_0\cdot a_2= \tfrac{1}{2^3}(a_0+a_2-a_{\rho^2})
 \end{array}
& \begin{array}{l}    
(a_0,a_1)=\tfrac{1}{2^6}\\
 (a_0,a_2)=\tfrac{1}{2^3} \\
 (a_0,a_{\rho^2})=\tfrac{1}{2^3}  
 \end{array} \\
 \hline
 5A & \begin{array}{l}a_{-2},\\ a_{-1},\\ a_0, \\a_1,\\ a_2, \\w_\rho \end{array}
&  \begin{array}{l}  
 a_0\cdot a_1= \tfrac{1}{2^7}(3a_0+3a_1-a_{-1}-a_2-a_{-2})+ w_\rho \\
 a_0\cdot a_2=\tfrac{1}{2^7}(3a_0+3a_2-a_1-a_{-1}-a_{-2})- w_\rho \\
 a_0\cdot w_\rho= \tfrac{7}{2^{12}}(a_1+a_{-1}-a_2-a_{-2})+ \tfrac{7}{2^{5}}w_\rho \\
 w_\rho\cdot w_\rho=\tfrac{5^2 \cdot 7}{2^{19}}(a_0+a_1+a_{-1}+a_2+a_{-2}) 

           \end{array}
& \begin{array}{l}     
 (a_0,a_1)=\tfrac{3}{2^7}\\
 (a_0,w_\rho)=0\\
  (w_\rho,w_\rho)=\tfrac{5^3 \cdot 7}{2^{19}}
        \end{array}  \\
 \hline
 6A & \begin{array}{l}
 a_{-2},\\ a_{-1},\\ a_0,\\
 a_1,\\ a_2,\\ a_3, \\
 a_{\rho^3},\\ u_{\rho^2} \end{array}
& \begin{array}{l}   
 a_0\cdot a_1= \tfrac{1}{2^6}(a_0+a_1-a_{-1}-a_2-a_{-2}-a_3+a_{\rho^3})+\tfrac{3^2\cdot5}{2^{11}} u_{\rho^2} \\
  a_0\cdot a_2= \tfrac{1}{2^5}(2a_0+2a_2 + a_{-2})-\tfrac{3^3\cdot5}{2^{11}} u_{\rho^2} \\
 a_0\cdot a_3=\tfrac{1}{2^3}(a_0+a_3-a_{\rho^3})\\
 a_0\cdot u_{\rho^2}= \tfrac{1}{3^{2}}(2a_0-a_2-a_{-2})+ \tfrac{5}{2^{5}}u_{\rho^2} \\
 a_{\rho^3}\cdot u_{\rho^2}=0

          \end{array}
& \begin{array}{l}    
(a_0,a_1)=\tfrac{5}{2^8}\\
 (a_0,a_2)=\tfrac{13}{2^8}\\
  (a_0,a_3)=\tfrac{1}{2^3}\\
  ( a_{\rho^3}, u_{\rho^2})=0

         \end{array}  \\
\hline

\end{array} 
 $$}
 \caption{The nine types of Norton-Sakuma algebras}
  \label{NS}
 \end{table}
 
   \begin{rem}
 \label{rem1} Note that, by the Norton-Sakuma Theorem, 
 dihedral algebras of type  $1A$, $2A$, $2B$, $3C$, and $4B$ are linearly spanned by axes, while for those of type $NA$, for $N\in\{3,4,5,6\}$, together with the axes a further vector (denoted by $u_\rho$, $v_\rho$, $w_\rho$, and $u_{\rho^2}$, respectively) is needed. \end{rem}
 
 The extra vectors $u_\rho$, $v_\rho$, $w_\rho$, resp.  $u_{\rho^2}$, above are called {\it $N$-axes}, for $N\in \{3,4,5\}$, resp. {\it $3$-axis }, for $N=6$, or simply  {\it odd axes}, when $N$ needs not to be specified. 
 
 For $N\in \{2,3,4,5\}$, denote by $V^{(NA)}$ the subspace of $V$ spanned by the $N$-axes (here the $2$-axes are the axes, i.e. the elements of $\A$). 
 Note that, by Table~\ref{NS}, the Norton-Sakuma algebras are $2$-closed. More generally, the $2$-closure of $V$, which we shall denote by $V^\circ$, is the linear span of the axes and the odd axes of $V$.

\begin{rem}
\label{remG}
 Let $\langle \langle a_0, a_1\rangle \rangle$ be a subalgebra of  the Griess algebra of type $NA$, with $N\in\{3,4,5,6\}$,  generated by two axes $a_0$ and $a_1$, and let  $\rho:=\tau_{a_0} \tau_{a_1}$. Then  the odd axis $u_\rho$ , $v_\rho$, $w_\rho$, or $u_{\rho^2}$, depends only on  the conjugacy class  of $\rho$ in the dihedral group $\langle \tau_{a_0}, \tau_{a_1}\rangle$. In particular, if $N\in\{3,4, 6\}$,  we have only one conjugacy class, if $N=5$ we have two conjugacy classes with representatives $\rho$ and $\rho^2$, respectively. In the latter case $w_{\rho^2}=-w_\rho$ (see~\cite[Introduction and Table 1]{N96}).
 \end{rem}
 
  This fact, together with the correspondence between axes and Miyamoto involutions is axiomatised  in the following conditions, usually known as Condition 2A and Axiom M8, respectively.
 \begin{enumerate}
 \item [(2A)]  For every pair of axes $a,b\in \Ax$, if $\tau_a\tau_b\in T$, then $\langle \langle a, b\rangle \rangle$ has type $2A$ and $\tau_a\tau_b=\tau_{c}$, where $c=a+b-8ab$.
 \item [(M8)] 
The vectors $a_\rho$, $a_{\rho^2}$, $a_{\rho^3}$ in the $2A$, $4B$, and $6A$ type algebras respectively, are Majorana axes in $V$ and, for $i\in \{1,2,3\}$,  $(\tau_{a_0}\tau_{a_1})^i=\tau_{a_{\rho^i}}\in T$ for $i\in \{1,2,3\}$. The vectors $u_\rho$, $v_{\rho}$, and $w_\rho$ in the $3A$, $4A$, and $5A$ type algebras depend uniquely only on  the conjugacy class  of $\rho$ in the dihedral group $\langle \tau_{a_0}, \tau_{a_1}\rangle$, and $w_{\rho^2}=-w_\rho$. \end{enumerate}

Let $(V , (\:, )_V, \cdot ,\Ax , G)$ be a Majorana representation of $G$ and suppose that $G$ contains a subgroup $H$ isomorphic to $S_4$ such that $V$ induces a Majorana representation $\hat V(H)$ on $H$ of shape $(2B, 3A)$. By~\cite[Section 5]{IPSS}, $\hat V(H)$ contains three idempotents  $\overline v_{g_i}$ ($i\in \{1,2,3\}$) of length $2$, uniquely depending (in $\hat V(H)$) on the subgroup $\langle g_i \rangle$ of order $4$ in $H$. Note that $\overline v_{g}$ is not a $4$-axis in $\hat V(H)$, since there is no dihedral subalgebra of type $4A$ in $\hat V(H)$. We call $\overline v_{g}$, a {\it \ghost $4$-axis}.
This situation is realized when $G=M$ is the Monster group, $V$ is the Griess algebra, and $H$ is a subgroup of $M$ isomorphic to $S_4$ such that the traspositions of $H$ are involutions of type $2A$ in $M$ and the bitraspositions of $H$ are elements of type $2B$ in $M$. Norton showed that, in this case, the vectors $\overline v_{g}$ in $\hat V(H)$ are true $4$-axes in $V$  depending uniquely on $\langle g\rangle$ (see~\cite[p.2462]{IPSS}). We shall therefore assume the following
\medskip 

\begin{enumerate}
    \item[(M8D)] \ghost $4$-axes $\overline v_{g}$ depend uniquely on the subgroup $\langle g \rangle$ of $ G$.
\end{enumerate}
\medskip

We call $(V , (\:, )_V, \cdot ,\Ax , G)$ a {\it diagonal Majorana representation} if, for every $a_1, a_2\in \Ax$ such that $|\tau_{a_1}\tau_{a_2}|=2$, $\langle \langle a_1, a_2\rangle \rangle $ is a Norton-Sakuma algebra of type $2B$ and Condition (M8D) holds.
\medskip

Given a Majorana algebra $V$, by the Norton-Sakuma Theorem, every subalgebra of $V$ generated by two distinct axes is an algebra of type $NL$, with $NL\in \{2A, 2B, 3A, 3C, 4A, 4B, 5A, 6A\}$. Clearly, the type $NL$ is constant on the $Aut(V)$-orbits of the $2$-generated subalgebras. The function that assigns the type $NL$ to each $Aut(V)$-orbit of the $2$-generated subalgebras of $V$ is the {\it shape} of $V$. 
\medskip

 \begin{lemma}
 \label{sh}
 Let  $(V, (\:, ), \cdot , \Ax, A_6)$ be a diagonal Majorana representation of $A_6$ with respect to the set $T$ of its bitranspositions and let $s,t\in T$. Then $|st| \in\{2,3,4,5\}$, moreover
 \begin{enumerate}
 \item if $st$ has order $2$, then $\langle \langle a_{s}, a_{t}\rangle \rangle $ has type $2B$;
 \item if $st$ has order $4$, then $\langle \langle a_{s}, a_{t}\rangle \rangle $ has type $4A$;
 \item if $st$ has order $5$, then $\langle \langle a_{s}, a_{t}\rangle \rangle $ has type $5A$.
 \end{enumerate}
 \end{lemma}
 
 \begin{proof}
If $st$ has order $2$, the result follows by the definition.

 Assume $st$ has order $4$ and suppose, by contradiction, that the algebra $\langle \langle a_{s}, a_{t}\rangle \rangle $ has type $4B$.  By Table~\ref{NS}, it follows that  its  subalgebra 
$\langle \langle a_{t}, a_{t}^s\rangle \rangle$
 is of type $2
A$, against  {\it (1)}.

 If $st$ has order $5$, the result follows by the Norton-Sakuma Theorem. 
 \end{proof}

 We close this section by listing some elementary properties of $A_6$ that we assume the reader is confident with and will be used throughout this paper without further reference.

\begin{proposition}\label{maximal}
$\:$
\begin{enumerate}
    \item All involutions in  $A_6$ are conjugate.
    \item The elements of order $5$ in $A_6$ part into two conjugacy classes which are closed under inversion.
    \item $A_6$ has five conjugacy classes of maximal subgroups: two classes isomorphic to $A_5$, two classes isomorphic to $S_4$, and one class isomorphic to $3^2:4$ (see e.g.~\cite{ATLAS}). 
    \item Every element of order $4$ in $A_6$ is contained in exactly two maximal subgroups isomorphic to $S_4$. In particular, if $t$ is an involution in $A_6$, there are exactly two maximal subgroups isomorphic to $S_4$ whose derived subgroups contain $t$.
    \item If $S$ is a subgroup of $A_6$ isomorphic to $S_4$, there are two subgroups isomorphic to $A_5$ containing the derived subgroup $S^\prime$ of $S$.
    \item Let $S$ be a subgroup of $A_6$ isomorphic to $S_4$ and let $s\in S\setminus S^\prime$ be an element of order $2$. Then there are an involution $r$ and an element $g$ of order $4$ in $S$ such that $C_S(s)=\{1, s, r, g^2\}$. 
\end{enumerate} 
\end{proposition}  
\section{Proof of Theorem~\ref{A7}} \label{Sec:start}

\begin{lemma}\label{6A}
Suppose a finite group $G$ has Majorana representation with respect to a set of generating involutions $T$ such that for every $r,s\in T$ such that $|rs|=2$, $\langle \langle a_r, a_s\rangle \rangle$ has type $2B$. Then, for every $r,s\in T$, $rs$ has order at most $5$. 
\end{lemma}
\begin{proof}
Let $r,s\in T$, by the theorem of Norton-Sakuma $rs$ has order at most $6$. Suppose $|rs|=6$. Then $|rsrsrs|$=2 and so by hypothesis $\langle \langle a_r, a_{srsrs}\rangle \rangle$ has type $2B$. On the other hand, the algebra $\langle \langle a_{r}, a_{s}\rangle \rangle$ is of type $6A$, and so by Axiom M8, its subalgebra $\langle \langle a_{r}, a_{srsrs}\rangle \rangle$ is of type $2A$, a contradiction. 
\end{proof}
\medskip

\begin{lemma}
\label{a5a6}
Let $z$ be the permutation $(1,7)(2,8)(3,9)(4,10)(5,11)(6,12)$ in $A_{12}$. 
There is a subgroup $H$ of $A_{12}$ isomorphic to $A_6$ such that 
\begin{enumerate}
\item  $C_{A_{12}}(z)=\langle z\rangle \times H$ 
\item  the involutions in $H$ are permutations of type $2^4$
\item  the  diagonal involutions in $\langle z\rangle \times H$ are  permutations of type $2^6$.
\end{enumerate}
\end{lemma}

\begin{proof}
Just take $H:=\{gg^z\: |\: g\in A_6\}$.
\end{proof}

\begin{proof}[Proof of Theorem~\ref{A7}] By~\cite[Table~4]{MS}, $A_5$ has a standard Majorana representation with shape $(2B, 3A, 5A)$.  By~\cite[Theorem~1]{A12}, there exists a (unique) saturated Majorana representation $\mathcal R$ of  $A_{12}$ (based on an embedding of $A_{12}$ in $M$).   By  Lemma~\ref{a5a6},  $\mathcal R$ induces a (standard) Majorana representation of $A_6$ with shape $(2B, 3A, 5A)$. Finally, by~\cite[Theorem 6]{A12}, the series 
$$A_8\leq A_8\times \langle(9,10)(11,12)\rangle \leq A_{12}\leq M$$
produces a (nonstandard) Majorana representation of $A_8$ with shape $(2B, 3A, 5A)$. 

Conversely, let $G=A_n$ and let $\mathcal R:=(V, (\:, \:)_V, \cdot, \Ax, G)$ be a Majorana representation of $G$, with respect to a $G$-invariant generating set $T$ of involutions of $G$, with shape $(2B, 3A, 5A)$. Since $A_3$ and $A_4$ are not generated by involutions we may assume $n\geq 5$. Assume by contradiction that the ``only if" assertion in (a) and (b) is false, we'll prove that in both cases we can find a pair of involutions in $T$ whose product has order $6$, against Lemma~\ref{6A}.  
Assume $\mathcal R$ is standard, if $n\geq 7$, take 
$$r:=(1,2)(3,4) \mbox{ and } s:= (1,5)(6,7),$$
then $r, s\in T$ and 
$|rs|=|(1,2,5)(3,4)(6,7)|=6.$ 
Now assume $\mathcal R$ is non standard, then $n\geq 8$, since, for $n\in\{5,6,7\}$ the involutions of $A_n$ are all bitranspositions. If $n\geq 9$, take  
$$r:=(1,2)(3,4)(5,6)(7,8) \mbox{ and } s:=(1,2)(3,5)(4,6)(7,9).$$
Then there exists a permutation $t$ fixing $\{1,\ldots,9\}$ such that $\{rt,st\}\subseteq T$. Thus 
$|rtst|=|rs|=|(3,6)(4,5)(7,8,9)|=6$. 
\end{proof}

\section{The inner products between the axes}
\label{Scalar}

 Given a group $G$ and $i\in \N$, let $G^{(i)}$ be a set of elements of order $i$ in $G$ which contains one representative from every subgroup of order $i$ in $G$, and these representatives are chosen in a such a way that, if $x,y\in G^{(i)}$, $\langle x\rangle $ and $\langle y\rangle$ are conjugate in $G$ if and only if $ x$ and $ y$ are conjugate in $G$ (note that $G^{(2)}=T$). Moreover, for a subgroup $H$ of $G$, $r,i\in \N$,  set
$$
H^{(i)}:=H\cap G^{(i)} \mbox{ and } H_r^{(i)}(g):=\{h\:|\:h\in H^{(i)}, |gh|=r\}.
$$

From now on we let $T$ be the set of the bitranspositions of $A_6$ and let 
$$
(V, (\:,\: )_V, \cdot , \Ax, A_6)
$$ 
be a Majorana representation of $A_6$ with respect to $T$, with shape $(2B, 3A, 5A)$, in particular
\begin{equation}\label{hyp}
\mbox{for every $s,t\in T$ such that $|st|=3$, the algebra $\langle \langle a_{t},a_{s}\rangle\rangle$ has type $3A$.}  
\end{equation}
 
\begin{rem}\label{positive}
By Axiom (M8), for $N\in\{2,3,4,5\}$, to each $\rho \in A_6^{(N)}$ there corresponds a unique  $N$-axis in $V$. We shall denote this axis by  $x_\rho$. By Remark~\ref{remG}, for $N=5$, if $\rho \in A_6^{(5)}$, then $\rho^2\not \in A_6^{(5)}$ and $w_{\rho^2}=-w_\rho $.
\end{rem}
In this case we call $w_\rho$ a {\it positive} $5$-axis and $w_{\rho^2}$ a {\it negative} $5$-axis.
 For $N\in \{3,4,5\}$, denote by $NY$ the set of $N$-axes in $V$, and set 
$$X:= \Ax\cup 3Y\cup 4Y \cup 5Y.$$ 
In this section we shall determine the inner products $(x,y)_V$ for $x,y\in X$.


\begin{rem}
\label{rem2}
By the linearity and the fact that $A_6$ acts as a group of isometries on $V$, in order to determine the inner products between the elements of $X$, we can reduce ourselves to determining the inner products of unordered pairs $(x_g,x_h)$ with $(\langle g\rangle,\langle h\rangle)$ varying in a set of representatives of their $A_6$-orbits (under the action induced by conjugation). Moreover, since the arguments we are using depend only on group theoretical properties that are $\aut(A_6)$-invariant, we may consider w.l.o.g. only those pairs in a set of representatives  of the  $\aut(A_6)$-orbits.   
\end{rem}

\begin{rem}
\label{rem3}
Let $t\in T$, $N\in \{3,4,5\}$, and $h\in A_6^{(N)}$ be such that $h^t=h^{-1}$. Then $\langle \langle a_t,x_h\rangle \rangle$  is a Norton-Sakuma algebra and the inner product $(a_t, x_h)_V$  is given in Table~\ref{NS}. This gives the first rows of Tables~\ref{tau3}, \ref{tau4}, and \ref{tau5}.
\end{rem} 

\begin{rem}\label{rem4}
A direct check in the set of subgroups of $A_6$ gives immediately all the columns except for the last  one in Tables~\ref{tau3}-\ref{tau44}. 
\end{rem}

\begin{lemma}\label{23}
Let  $t\in T$ and $h\in A_6^{(3)}$. The possible isomorphism classes for $\langle t, h\rangle$ and the corresponding values of the inner products $(a_{t}, u_h)_V$  are those  listed in the second and third column of Table~\ref{tau3} respectively.
\end{lemma}
\begin{proof}
 If $\langle t, h\rangle \cong S_4$, the result follows by~\cite[Table 11]{IPSS}. In the remaining two cases, note that, by~\cite[Table~4]{MS} there is a unique standard Majorana representation $\mathcal R$ of $A_5$ with shape $(2B, 3A, 5A)$. By Lemma~\ref{a5a6}, $\mathcal R$ is induced by the saturated Majorana representation of $A_{12}$. So the values of the inner products are those given  in~\cite[Table 6, rows 4 and 9]{A12}. 
\end{proof}

\begin{table}
$$
\begin{array}{|c|c|c|}
\hline
 |th| & \langle t, h\rangle & (a_{t}, u_{h})_V\\
\hline
 2 & S_3 & \tfrac{1}{4}\\
\hline
 4 & S_4 & \tfrac{13}{180}\\
\hline 
 3 &  A_4&\tfrac{2}{45}\\
\hline 
 5 & A_5 & \tfrac{1}{30}\\
\hline
\end{array}
$$

\caption{The inner products between $2$-axes and $3$-axes}\label{tau3}
\end{table}


\begin{lemma}\label{24}
Let  $t\in T$ and $g\in A_6^{(4)}$. The possible isomorphism classes for $\langle t, g\rangle$ and the corresponding values of the inner products $(a_{t}, v_g)_V$  are those  listed in the second and  the third column of  Table~\ref{tau4} respectively.
\end{lemma}

\begin{proof}
We shall proceed as follows. 
For each possibility of $\langle t, g \rangle$ listed in Table~\ref{tau4} we choose $r,s\in T$ such that $g=rs$. Then, by Lemma~\ref{sh},  $\langle \langle a_{r}, a_{s} \rangle \rangle$ is a Norton-Sakuma algebra of type $4A$, so, by Table~\ref{NS}, 
%
%
%
$$
v_g=a_{r}+a_{s}+\tfrac{1}{3}(a_{rsr}+a_{srs})-\tfrac{2^6}{3}a_{r}a_{s}.$$ 
Thus 
\begin{align}
\label{equita}
(a_{t}, v_g)_V=& (a_{t}, a_{r})_V+(a_{t},a_{s})_V+\tfrac{1}{3}[(a_{t},a_{rsr})_V+(a_{t},a_{srs})_V]-\tfrac{2^6}{3}(a_{t},a_{r}a_{s})_V \nonumber\\
=& (a_{t}, a_{r})_V+(a_{t},a_{s})_V+\tfrac{1}{3}[(a_{t},a_{rsr})_V+(a_{t},a_{srs})_V]-\tfrac{2^6}{3}(a_{t}a_{r},a_{s})_V.
\end{align}
Now, for each choice of $r$ and $s$,  all the above inner products, except for $(a_{t}a_{r},a_{s})_V$, are given by  Table~\ref{NS} and so we are reduced to compute, in each case, 
the value of 
$(a_{t}a_{r},a_{s})_V$. 

Assume first $\langle t,g\rangle \cong A_6$ and 
let 
$$
{\mathcal H}:=\{(\langle t\rangle ,\langle g\rangle )\:|\: t,g \in A_6, \, |t|=2,\, |g|=4,\mbox{ and}  \, \langle t,g\rangle=A_6\}.
$$ 
Then $\mathcal H$  parts into two orbits under the action of $A_6$ (via conjugation) with representatives $(\langle (3,4)(5,6)\rangle, \langle (1,6)(2,5,3,4)\rangle)$ and $(\langle (3,4)(5,6)\rangle, \langle(1,5)(2,6,3,4)\rangle )$ respectively.
Suppose
$$
t=(3,4)(5,6) \mbox{  and } g=(1,6)(2,5,3,4),
$$ 
and write $g=rs$ with $r:=(3,4)(2,5)$ and $s:=(2,3)(1,6)$. Since $|tr|=3$, by~(\ref{hyp}),  $\langle \langle a_{t}, a_{r}\rangle \rangle $ is a Norton-Sakuma algebra of type $3A$, thus, by Table~\ref{NS} 
\begin{equation}
\label{req}
a_{t}a_{r}=\tfrac{1}{32}(2a_{t} +2a_{r}+a_{trt}) - \tfrac{135}{2^{11}}u_{tr}.
\end{equation} 
The value of $(a_{t}a_{r},a_{s})_V$ is obtained by a direct computation using Lemma~\ref{23} and Table~\ref{NS}.  Swapping $5$ and $6$,  the result for $(\langle t\rangle ,\langle g\rangle )= (\langle (3,4)(5,6)\rangle, \langle(1,5)(2,6,3,4)\rangle )$ is obtained in the same way, taking $r=(3,4)(2,6)$ and $s=(2,3)(1,5)$.

Next assume  $\langle t,g\rangle \cong 3^2:4$. As in the previous case,  we may w.l.o.g. assume  that  
$$
t=(3,4)(5,6)\mbox{ and }g=(1,2,3,5)(4,6).
$$
Let $s:=(1,3)(4,6)$ and $r:=(1,5)(2,3)$. As in the previous case, by Equation~(\ref{equita}) and Table~\ref{NS}, we get  
$$
(a_{t},v_g)_V= \tfrac{3}{128}+ \tfrac{13}{256}+ \tfrac{1}{3}(\tfrac{1}{32} +\tfrac{3}{128})-\tfrac{2^6}{3}(a_{t}a_{r}, a_{s})_V.
$$
Since $|tr|=3$,  $\langle \langle a_{t}, a_{r}\rangle \rangle $ is of type $3A$, the result follows by substituting the expression for $a_{t}a_{r}$ given in Equation~(\ref{req}) and using Lemma~\ref{23}. 

The inner products for the last two cases are obtained in a similar way (see also~\cite[Lemma~4.12]{Marta}). 
\end{proof}

\begin{table}
$$
\begin{array}{|c|c|c|}
\hline
 |tg| & \langle t, g\rangle & (a_{t}, v_{g})_V\\
\hline
 2 & D_8 & \tfrac{3}{8}\\
\hline
 5 &A_6 & \tfrac{5}{128}\\
\hline 
 4 &3^2:4 & \tfrac{27}{256}\\
\hline
 4 & 4 & 0\\
\hline 
 3 & S_4&\tfrac{5}{64}\\
\hline
\end{array}
$$

\caption{The inner products between $2$-axes and $4$-axes}\label{tau4}
\end{table}

\begin{table}
$$
\begin{array}{|c|c|c|}
\hline
  |tf| & \langle t, f\rangle & (a_{t}, w_{f})_V\\
\hline
 2 & D_{10} & 0\\
\hline
 5 & A_5 & -\tfrac{7}{2^{14}}\\
\hline
 3 & A_5 & \tfrac{7}{2^{14}}\\
\hline 
 5 & A_6&\tfrac{7}{2^{12}}\\
\hline 
 4 &A_6 & -\tfrac{7}{2^{12}}\\
\hline
\end{array}
$$

\caption{The inner products between $2$-axes and $5$-axes}\label{tau5}
\end{table}

\begin{lemma}\label{25}
Let  $t\in T$ and $f\in A_6^{(5)}$. The possible isomorphism classes for $\langle t, g\rangle$ and the corresponding values of the inner products $(a_{t}, w_f)_V$  are those  listed in the second and  the third column of  Table~\ref{tau5} respectively.
\end{lemma}
\begin{proof}
We argue as in the proof of Lemma~\ref{24}: we write $f$ as the product of two elements $r$ and $s$ of $T$. Then, by the Theorem of Norton-Sakuma, the algebra $\langle \langle a_{r}, a_{s}\rangle \rangle $ has type $5A$, so, by Table~\ref{NS}, 
$$ 
w_f=-\tfrac{1}{2^7}(3a_{r}+ 3a_{s} -a_{srs} -a_{rsr} -a_{r(sr)^3}) + a_{r}a_{s},
$$
whence 
\begin{eqnarray*}
(a_{t}, w_f)_V&=&-\tfrac{1}{2^7}[3(a_{t},a_{r})_V+ 3(a_{t},a_{s})_V -(a_{t},a_{srs})_V -(a_{t},a_{rsr})_V \\
&&-(a_{t},a_{r(sr)^3})_V] + (a_{t},a_{r}a_{s})_V\\
&=&-\tfrac{1}{2^7}[3(a_{t},a_{r})_V+ 3(a_{t},a_{s})_V -(a_{t},a_{srs})_V -(a_{t},a_{rsr})_V \\
&&-(a_{t},a_{r(sr)^3})_V] + (a_{t}a_{s}, a_{r})_V.
\end{eqnarray*}
So, this time, we are reduced to compute the value of  
$(a_{t}a_{s},a_{r})_V$. 

Assume first that  $|tf|=5$ and $\langle t,f\rangle\cong  A_5$, as in line $2$ of Table~\ref{tau5}. W.l.o.g., we may assume $t=(3,4)(5,6)$ and $f=(2,5,3,6,4)$ and choose $r:=(2,4)(5,6)$ and $s:=(4,5)(3,6)$, so that $f=rs$. Since $|st|=2$, by Lemma~\ref{sh}, $\langle \langle a_{t}, a_{s}\rangle \rangle $ is a Norton-Sakuma algebra of type $2B$, whence, by Table~\ref{NS},  $a_{t}a_{s}=0$. Thus, $(a_{t}a_{s}, a_{r})_V=(0, a_{r})_V=0$. 

  Suppose that  $|tf|=3$ and $\langle t,f\rangle\cong  A_5$, as in line $3$ of Table~\ref{tau5}.
Then the pair $(t,f^2)$ satisfies the conditions of line $2$, thus the result follows, since, by Remark~\ref{remG}, $w_f=-w_{f^2}$. 

  Suppose that  $|tf|=5$ and $\langle t,f\rangle\cong  A_6$, as in line $4$ of Table~\ref{tau5}.
 Then we may assume $t=(3,4)(5,6)$ and $f=(1,2,3,4,5)$ and choose $r:=(2,4)(1,5)$ and $s:=(2,5)(3,4)$, so that $f=rs$. Proceeding as in the previous case and using Lemma~\ref{23} we get the value of the inner product. 

Finally,  if $|tf|=3$ and $\langle t,f\rangle\cong  A_6$, as in line $5$ of Table~\ref{tau5}, the pair $(t,f^2)$ satisfies the conditions of line $4$ and we conclude as in the previous case. 
\end{proof}

In order to treat the cases involving two odd-axes, we use three easy observations, the first of which describes a standard procedure in Majorana theory (see e.g.~\cite{IS12, A67}).

\begin{rem}\label{start}
Let $h,k\in A_6$ such that $x_h$ and $x_k$ are odd axes of $V$. Suppose there is $t\in T$ inverting by conjugation both $h$ and $k$. Then the inner product $(x_h, x_k)_V$ can be computed in the following way. The algebra 
$\langle \langle a_{t}, x_h\rangle\rangle$ is contained in a Norton-Sakuma algebra of type $|h|A$ and similarly $\langle \langle a_{t},x_k\rangle\rangle$ is contained in a Norton-Sakuma algebra of type $|k|A$. Let  $e_h$ be a $0$-eigenvector for $\ad_{a_{t}}$ and $e_k$ be a  $\tfrac{1}{4}$-eigenvector for $\ad_{a_{t}}$. Using~\cite[Table 4]{IPSS}, we  express $e_h$ (resp. $e_k$) as a linear combination of axes and $x_h$ (resp. $x_k$). Since $e_h$ and $e_k$ are eigenvectors relative to different eigenvalues, we get $$(e_h, e_k)_V=0.$$ Thus, expanding this equation by substituting $e_h$ and $e_k$ with their respective linear combinations, by Lemmas~\ref{23}, \ref{24}, and \ref{25}, we can obtain the value for $(x_h, x_k)_V$.
\end{rem}

\begin{lemma}\label{ob}
Let $A$ be a Majorana algebra and let $a_i$ be axes for $i\in \{1,2,3\}$. If $e\in A$ is a $0$-eigenvector for $\ad_{a_2}$, then 
$$
(a_1 a_2, e a_3)_A=(a_1 e, a_2 a_3)_A.
$$
\end{lemma}
\begin{proof}
By the associativity of the inner product and by Lemma~1.10 in~\cite{IPSS} we have
\begin{align*}
(a_1a_2, ea_3)_A=&((a_1a_2)e,a_3)_A=(a_2(a_1e),a_3)_A=(a_1e,a_2a_3)_A.
\end{align*}
\end{proof}

Note that Norton inequality and positive definiteness of the Majorana form imply that in any Majorana algebra two orthogonal idempotents multiply to zero. If one of the two idempotents is a $2$-axis, this can be proved independently from Norton inequality. 

\begin{lemma}\label{idem}
Let $A$ be a Majorana algebra, let $a$ be an axis and let $e$ be an idempotent. If $(a,e)_A=0$, then $a\cdot e=0$.
\end{lemma}
\begin{proof}
Decompose $e$ as the sum of a $1$-eigenvector $e_1$, a $0$-eigenvector $e_0$, a $\tfrac{1}{4}$-eigenvector $e_\al$, and a $\tfrac{1}{32}$-eigenvector  $e_\bt$ for $ad_a$. Since $(a,e)_A=0$ and eigenvectors relative to different eigenvalues are orthogonal, we get, by the associativity,   
\begin{eqnarray*}
0&=&(a,e)_A=(a,e^2)_A=(a\cdot e, e)_A=(e_1+\tfrac{1}{4}e_\al+\tfrac{1}{32}e_\bt, e_1+e_0+  e_{\al}+e_\bt)_A\\
&=& 
(e_1,e_1)_A+\tfrac{1}{4}(e_\al, e_\al)_A+\tfrac{1}{32}(e_\bt,e_\bt)_A.
\end{eqnarray*} 
By the positive definiteness of the form, it follows  $e_1=e_\al=e_\bt=0$, $e=e_0$ and so $a\cdot e=0$.
\end{proof}


\begin{lemma}\label{33}
Let  $h,k\in A_6^{(3)}$, $g\in A_6^{(4)}$, $f\in A_6^{(5)}$. 
\begin{enumerate}
\item The values of the inner products $(u_h, u_k)_V$ are given in the last column of Table~\ref{tau33}.
\item The values of the inner products $(u_h, v_g)_V$ are given in the last column of Table~\ref{tau34}.
\item The values of the inner products $(u_{h}, w_{f})_V$ are given in the last column of Table~\ref{tau35}.
\end{enumerate}
\end{lemma}

\begin{table}
$$
\begin{array}{|c|c|}
\hline
\langle h, k\rangle & (u_{h}, u_{k})_V\\
\hline
\hline
 3 & \tfrac{8}{5}\\
\hline
 3\times 3& 0\\
\hline
 A_4 & \tfrac{56}{3^35^2}\\
\hline 
 A_5&\tfrac{208}{3^45^2}\\
\hline 
A_6 & \tfrac{256}{3^45^2}\\
\hline
\end{array}
$$

\caption{The inner products between two $3$-axes}\label{tau33}
\end{table}
\begin{proof}
For every $h,k\in A_6^{(3)}$, $g\in A_6^{(4)}$, and $f\in A_6^{(5)}$ there exists $t\in T$ inverting by conjugation both $h$ and $k$, or $h$ and $g$, or $h$ and $f$. The result then follows using the argument described in Remark~\ref{start}. \end{proof}

\begin{table}
$$
\begin{array}{|c|c|c|c|}
\hline
 |hg| &  |hg^{-1}| &\langle h, g\rangle & (u_{h}, v_{g})_V\\
\hline
\hline
2 & &S_4 & \tfrac{1}{9}\\
\hline
 4& &3^2:4 & \tfrac{3}{20}\\
\hline
 3&5 & A_6 & \tfrac{2}{9}\\
\hline 
 5&5&A_6&\tfrac{1}{12}\\
\hline 
\end{array}
$$

\caption{The inner products between $3$-axes and $4$-axes}\label{tau34}
\end{table}

%

\begin{table}
$$
\begin{array}{|c|c|c|c|}
\hline
 |hf| &|hf^{-1}|&\langle h, f\rangle & (u_{h}, w_{f})_V\\
\hline
\hline
 2& 5&A_5 & \tfrac{49}{2^{9}3^2 5}\\
\hline
 5& 3&A_5 &-\tfrac{49}{2^{9}3^2 5}\\
\hline
 3 & 4&A_6 & -\tfrac{91}{2^{9}3^2 5}\\
\hline 
 4& 5&A_6& \tfrac{91}{2^{9}3^2 5}\\
\hline 
\end{array}
$$

\caption{The inner products between $3$-axes and  $5$-axes}\label{tau35}
\end{table}

\begin{lemma}\label{44}
Let  $g_1, g_2\in A_6^{(4)}$. The values of the inner products $(v_{g_1}, v_{g_2})_V$ are given in the last column of Table~\ref{tau44}.
\end{lemma}

\begin{table}
$$
\begin{array}{|c|c|c|c|}
\hline
 |g_1g_2| & |g_1g_2^{-1}| &\langle g_1, g_2\rangle & (v_{g_1}, v_{g_2})_V\\
\hline
\hline
 2&1& 4 &2\\
\hline
3 &3&S_4 & \tfrac{11}{48}\\
\hline
 2&3&3^2:4 & \tfrac{53}{256}\\
\hline
 5&5&A_6&\tfrac{5}{48}\\
\hline 
 5 & 4&A_6 & \tfrac{89}{384}\\
\hline 
\end{array}
$$

\caption{The inner products between two $4$-axes}\label{tau44}
\end{table}
\begin{proof}
If $g_1$ and $g_2$ are as in the first row of Table~\ref{tau44}, then $g_1=g_2$, whence, by Table~\ref{NS}, $(v_{g_1}, v_{g_1})_V=2$. 

Suppose $\langle g_1, g_2\rangle \cong S_4$, as in the second row of Table~\ref{tau44}. Then we may assume $g_1=(1,2)(3,6,4,5)$ and $g_2=(1,2)(3,6,5,4)$. Let $t=g_2^2$. Then, by Table~\ref{tau4} $(a_t, v_{g^2})_V=0$ and so, by Lemma~\ref{idem}, $v_{g_2}$ is a $0$-eigenvector for $\ad_{a_t}$. Furthermore, since $t$ inverts $g_1$, $tg_1$ is an involution whence $\langle \langle a_{t}, a_{tg_1}\rangle\rangle$ is a Norton-Sakuma algebra of type $4A$ containing $v_{g_1}$. By~\cite[Table 4]{IPSS}, 
\begin{equation}
\label{vap}
e:=v_{g_1}-\tfrac{1}{3}a_{t} -\tfrac{2}{3}(a_{tg_1}+a_{tg_1^3})-\tfrac{1}{3}a_{t^{g_1}}
\end{equation}  is a $\tfrac{1}{4}$-eigenvector for $\ad_{a_t}$ in this algebra. As in Remark~\ref{start}, it follows that $(v_{g_2}, e)_V=0$ and, expanding this equation by substituting the expression for $e$ given in Equation~(\ref{vap}), we get 
$$
(v_{g_1}, v_{g_2})_V=\tfrac{11}{48}.
$$ 

Now suppose   $\langle g_1, g_2\rangle \cong 3^2:4$, as in line $3$ of Table~\ref{tau44}. W.l.o.g., we may assume $g_1=(1,2)(3,5,4,6)$ and $g_2=(1,4)(2,6,3,5)$. Let $t_1=(1,2)(3,4)$, $r_1=(3,6)(4,5)$, $t_2=(1,4)(2,3)$, and $r_2=(2,5)(3,6)$ so that $g_1=t_1r_1$ and $g_2=t_2r_2$. Then, by Table~\ref{NS}, for $  i\in \{1,2\}$, we have
$$
v_{g_i}=a_{t_i}+a_{r_i}+\tfrac{1}{3}(a_{r_it_ir_i}+a_{t_ir_it_i})-\tfrac{64}{3}a_{t_i}a_{r_i}
$$
whence, by the linearity of the inner product and Lemmas~\ref{23},~\ref{24},~\ref{25}, and~\ref{33}, the computation of $(v_{g_1}, v_{g_2})_V$ reduces to that of $(a_{t_1}a_{r_1}, a_{t_2}a_{r_2})_V$. Now, $\langle \langle a_{t_1}, a_{t_2}\rangle \rangle$ is of type $2B$ and so $a_{t_2}$ is a $0$-eigenvector for $\ad_{a_{t_1}}$. Hence, by Lemma~\ref{ob}, we have
$$
(a_{t_1}a_{r_1}, a_{t_2}a_{r_2})_V=(a_{r_1}a_{t_2},a_{t_1}a_{r_2})_V.
$$
Since $|r_1t_2|=|t_1r_2|=3$, by~(\ref{hyp}), the two algebras $\langle \langle a_{r_1}, a_{t_2}\rangle \rangle$ and $\langle \langle a_{t_1},a_{r_2}\rangle \rangle$ are of type $3A$. By Table~\ref{NS}, the vectors $a_{r_1}a_{t_2}$ and $a_{t_1}a_{r_2}$ are linear combinations of axes and $3$-axes, whence, using Lemmas~\ref{23} and~\ref{33}, a direct computation gives 
$$
(a_{r_1}a_{t_2},a_{t_1}a_{r_2})_V=\tfrac{151}{2^{19}}\: \mbox{ and }\: (v_{g_1}, v_{g_2})_V=\tfrac{53}{256}.
$$ 

Suppose $\langle g_1, g_2\rangle \cong A_6$ and $|g_1g_2|=|g_1g_2^{-1}|=5$, as in the fourth row of Table~\ref{tau44}. Then we may assume $g_1=(1,2)(3,6,4,5)$ and $g_2=(1,2,3,4)(5,6)$. Since $t:=(1,2)(3,4)$ inverts both $g_1$ and $g_2$, by Remark~\ref{start} we get
$$
(v_{g_1}, v_{g_2})_V=\tfrac{5}{48}.
$$

Finally, suppose $|g_1g_2|=5$, $|g_1g_2^{-1}|=4$ and $\langle g_1, g_2\rangle\cong A_6$, as in the fifth row of Table~\ref{tau44}. Then we may assume $g_1=(1,2)(3,5,4,6)$ and $g_2=(1,6)(2,5,3,4)$. Set 
$$
t_1:=(3,5)(4,6), \:\:\:r_1:=(1,2)(3,4), \:\:\:t_2=(2,5)(3,4), \:\:\mbox{ and } \:\: r_2:=(1,6)(2,3)
$$ 
so that $g_1=t_1r_1$ and $g_2=t_2r_2$. 
By Table~\ref{NS}, we have
\begin{align} \label{ex}
(v_{g_1}, v_{g_2})_V=&(a_{t_1}, v_{g_2})_V+(a_{r_1}, v_{g_2})_V +\tfrac{1}{3}(a_{t_1r_1t_1}, v_{g_2})_V+\tfrac{1}{3}(a_{r_1t_1r_1}, v_{g_2})_V\\
&-\tfrac{64}{3}(a_{t_1}a_{r_1},  v_{g_2})_V \nonumber
\end{align}
The first four summands  of Equation~(\ref{ex}) are inner products between axes and $v_{g_2}$, so by Lemma~\ref{24}, 
\begin{align}
\label{ex1}
(v_{g_1}, v_{g_2})_V=& \tfrac{5}{2^7}+\tfrac{5}{2^7}+\tfrac{3^2}{2^8}+\tfrac{3^2}{2^8} - \tfrac{64}{3}(a_{t_1}a_{r_1},  v_{g_2})_V= \tfrac{19}{2^7} - \tfrac{64}{3}(a_{t_1}a_{r_1},  v_{g_2})_V.
\end{align} 
In order to compute the value of  
$(a_{t_1}a_{r_1},  v_{g_2})$, observe that $\langle \langle a_{t_1}, a_{t_2}\rangle \rangle$ is of type $5A$. 
Taking the difference multiplied by $8$ of the two basis vectors of the $0$-eigenspace for  $\ad_{a_{0}}$ for the Norton-Sakuma algebra $\langle \langle a_{0}, a_{1}\rangle \rangle$ of type $5A$ given in~\cite[Table 4]{IPSS}, and evaluating it for $a_0=a_{t_2}$ and $a_1=a_{t_1}$,  we get that    
\begin{equation}\label{e}
e:=-\tfrac{3}{32}a_{t_2}+a_{t_1}+a_{t_1t_2t_1}+a_{t_2t_1t_2}+a_{t_1t_2t_1t_2t_1}
\end{equation} 
is a $0$-eigenvector for $\ad_{a_{t_2}}$ in $\langle \langle a_{t_1}, a_{t_2}\rangle \rangle$. 
By Equation~(\ref{e}) we obtain 
\begin{align}\label{eqe}
(ea_{r_1}, v_{g_2})_V=& -\tfrac{3}{32}(a_{t_2}a_{r_1},  v_{g_2})_V+(a_{t_1}a_{r_1},  v_{g_2})_V+(a_{t_1t_2t_1}a_{r_1},  v_{g_2})_V\\
&+(a_{t_2t_1t_2}a_{r_1},  v_{g_2})_V+(a_{t_1t_2t_1t_2t_1}a_{r_1},  v_{g_2})_V, \nonumber 
\end{align}
whence
\begin{align}\label{eqe2}
(a_{t_1}a_{r_1},  v_{g_2})_V=&(ea_{r_1}, v_{g_2})_V+\tfrac{3}{32}(a_{t_2}a_{r_1},  v_{g_2})_V-(a_{t_1t_2t_1}a_{r_1},  v_{g_2})_V\\
&-(a_{t_2t_1t_2}a_{r_1},  v_{g_2})_V-(a_{t_1t_2t_1t_2t_1}a_{r_1},  v_{g_2})_V. \nonumber
\end{align}
Now $a_{t_2}a_{r_1}\in \langle \langle a_{t_2}, a_{r_1}\rangle \rangle$ and $a_{t_1t_2t_1t_2t_1}a_{r_1}\in \langle \langle a_{t_1t_2t_1t_2t_1},a_{r_1}\rangle \rangle$. These are Norton-Sakuma algebras of type $3A$ (since $|t_2r_1|=3$ and $|t_1t_2t_1t_2t_1r_1|=3$) and so, by Table~\ref{NS}, the products $a_{t_2}a_{r_1}$ and $a_{t_1t_2t_1t_2t_1}a_{r_1}$ are linear combinations of axes and $3$-axes.  Thus, by Lemma~\ref{23} and Lemma~\ref{33} we obtain
\begin{equation}
\label{exx3}(a_{t_2}a_{r_1},  v_{g_2})_V=\tfrac{51}{2^{12}}\:\: \mbox{ and }\:\:(a_{t_1t_2t_1t_2t_1}a_{r_1},  v_{g_2})_V=-\tfrac{13}{2^{12}}.
\end{equation}
The algebra $ \langle \langle a_{t_1t_2t_1}, a_{r_1}\rangle \rangle$ is of type $4A$ and  $v_{(1,2,4,3)(5,6)}$ is a $4$-axis in it. Hence the product $a_{t_1t_2t_1} a_{r_1}$ is a linear combination of axes and $v_{(1,2,4,3)(5,6)}$. Since the pair $(v_{(1,2,4,3)(5,6)}, v_{g_2})$ satisfies the conditions in the third row  of Table~\ref{tau44}, by the previous case we get  
\begin{equation}
\label{exx4}
(a_{t_1t_2t_1}a_{r_1},  v_{g_2})_V=-\tfrac{3^2}{2^{13}}.
\end{equation}
The algebra $\langle \langle a_{t_2t_1t_2}, a_{r_1}\rangle \rangle$ is of type $5A$ with $5$-axis $w_{(1,2,3,6,4)}$, whence the product $a_{t_2t_1t_2}a_{r_1}$ is a linear combination of axes and $w_{(1,2,3,6,4)}$. Since $(1,6)(2,3)$ inverts both $(1,2,3,6,4)$ and $g_2$, by Remark~\ref{start}, we get 
$$(w_{(1,2,3,6,4)}, v_{g_2})_V=\tfrac{7\cdot13}{2^{14}},$$ 
whence 
\begin{equation}
\label{exx5}
(a_{t_2t_1t_2}a_{r_1}, v_{g_2})_V=\tfrac{23}{2^{13}}.
\end{equation} 
Thus, by Equations~(\ref{eqe2}),  (\ref{exx3}), (\ref{exx4}), and (\ref{exx5}), we obtain
\begin{align}
\label{eqe3} 
(a_{t_1}a_{r_1},  v_{g_2})_V=&(ea_{r_1}, v_{g_2})_V+\tfrac{3\cdot 5\cdot 23}{2^{17}}.
\end{align}
To compute the value of $(ea_{r_1}, v_{g_2})_V$, observe that,
 by Table~\ref{NS}, 
$$
v_{g_2}=a_{t_2}+a_{r_2}+\tfrac{1}{3}(a_{r_2t_2r_2}+a_{t_2r_2t_2})-\tfrac{64}{3}a_{t_2}a_{r_2}
$$ 
and so
\begin{align}\label{eqseconda}
(ea_{r_1}, v_{g_2})_V=&(ea_{r_1}, a_{t_2})_V+(ea_{r_1}, a_{r_2})_V\\
&+\tfrac{1}{3}(ea_{r_1}, a_{r_2t_2r_2})_V+\tfrac{1}{3}(ea_{r_1}, a_{t_2r_2t_2})_V
-\tfrac{64}{3}(ea_{r_1}, a_{t_2}a_{r_2})_V.\nonumber
\end{align}
Now, by Equation~(\ref{e}), we obtain 
$$ea_{r_1}= -\tfrac{3}{32}a_{t_2}a_{r_1}+a_{t_1}a_{r_1}+a_{t_1t_2t_1}a_{r_1}+a_{t_2t_1t_2}a_{r_1}+a_{t_1t_2t_1t_2t_1}a_{r_1}$$ which  is a linear combination of axes and odd axes.  Thus, substituting the above expression for $ea_{r_1}$ in Equation~(\ref{eqseconda}), by Lemmas~\ref{23},~\ref{24}, and \ref{25}, we get 
\begin{equation}
\label{penultima}
(ea_{r_1}, v_{g_2})_V=-\tfrac{4009}{2^{18}\cdot 3} -\tfrac{64}{3}(ea_{r_1}, a_{t_2}a_{r_2})_V
\end{equation}
Furthermore, by Lemma~\ref{ob}, we have 
$$
(ea_{r_1}, a_{t_2}a_{r_2})_V=(a_{r_1} a_{t_2}, ea_{r_2})_V.
$$
Since $a_{r_1} a_{t_2}$ is a linear combination of axes and a $3$-axis while $ea_{r_2}$, similarly as above, is a linear combination of axes and  odd axes, using Lemmas~\ref{23}, \ref{24}, \ref{25}, and~\ref{33} we get $(a_{r_1} a_{t_2}, ea_{r_2})_V=\tfrac{1133}{2^{24}}$, whence, by Equation~(\ref{penultima}),  
$$
(ea_{r_1}, v_{g_2})_V=-\tfrac{2053}{2^{17}\cdot 3}.
$$  Substituting this value in Equation~(\ref{eqe3}), we get  
$$
(a_{t_1} a_{r_1}, v_{g_2})_V= -\tfrac{1}{2^8}
$$
and the result follows by Equation~(\ref{ex1}).
\end{proof}

\begin{lemma}\label{45}
Let  $g\in A_6^{(4)}, f\in A_6^{(5)}$. The values of the inner products $(v_{g}, w_{f})_V$ are given in Table~\ref{tau45}. 
\end{lemma}
\begin{table}
$$
\begin{array}{|c|c|c|}
\hline
 |gf|& |gf^{-1}|  & (v_{g}, w_f)_V\\
\hline
\hline
 5& 3 & -\tfrac{91}{2^{14}}\\
\hline
4& 3&\tfrac{91}{2^{14}}\\
\hline
 5&5& 0\\
\hline 
 5&4&-\tfrac{7}{2^{11}}\\
\hline
 5 & 2 & \tfrac{7}{2^{11}}\\
\hline 
\end{array}
$$
\caption{The inner products between $4$-axes and $5$-axes}\label{tau45}
\end{table}
\begin{proof}
Suppose $|gf|=5$ and $|gf^{-1}|=3$, as in the first row 
of Table~\ref{tau45}. We may assume $g=(1,6)(2,5,3,4)$ and $f=(1,3,4,2,6)$. Since these two permutations are inverted by $(1,6)(2,3)$, the result follows by Remark~\ref{start}. 
Moreover, since $|gf^2|=4$ and $|gf^{-2}|=3$, the value in the second row of Table~\ref{tau45} is obtained by Remark~\ref{remG}, substituting $f$ by $f^2$.

Next suppose $|gf|=5$ and $|gf^{-1}|=5$, as in the third row of Table~\ref{tau45}. We may assume $g=(1,2)(3,4,5,6)$ and $f=(2,3,4,6,5)$. Set $t:=(3,5)(4,6)=g^2$. Then, by the fourth row of  Table~\ref{tau4}, $(a_{t},v_g)_V=0$, whence by Lemma~\ref{idem}, 
\begin{equation}
    \label{0eig}
    \ad_{a_t}(v_g)=0.
\end{equation} Since $f$ is inverted by $t$, $\langle t,f\rangle$ is a dihedral group of order $10$, whence, by Lemma~\ref{sh}, $\langle \langle a_{t}, a_{tf}\rangle \rangle$ is a Norton-Sakuma algebra of type $5A$. Let 
\begin{equation} 
\label{e2} e:=\tfrac{1}{2^7}(a_{tf}+a_{ft}-a_{tf^2}-a_{tf^3})+w_f.
\end{equation}
By~\cite[Table~4]{IPSS}, $e$ is a  $\tfrac{1}{4}$-eigenvector for $\ad_{a_t}$. Thus, by Equation~(\ref{0eig}), we get 
\begin{equation}
    \label{evg}
    (e, v_g)_V=0.
\end{equation} 
Since, by Lemma~\ref{24}, $(a_{tf}+a_{ft}-a_{tf^2}-a_{tf^3}, v_g)_V=0$, substituting in Equation~(\ref{evg}) the expression of $e$ given in Equation~(\ref{e2}),  we obtain $(v_g, w_f)_V=0$. 

Suppose $|gf|=5$ and $|gf^{-1}|=4$. We may assume $g=(1,2)(3,6,5,4)$ and $f=(1,3,6,2,4)$. 
Set $t_1:=(3,6)(4,5)$, $r_1:=(1,2)(3,5)$, $t_2:=(1,2)(3,6)$, and $r_2:=(1,4)(2,3)$ so that $g=t_1r_1$ and $f=t_2r_2$. Then, by Lemma~\ref{sh}, the Norton-Sakuma algebra $\langle \langle  a_{t_1}, a_{r_1}\rangle \rangle$ (resp.  $\langle \langle  a_{t_2}, a_{r_2}\rangle \rangle$) is of type $4A$ (resp. $5A$), and contains the odd axis $v_g$ (resp. $w_g$). Thus, by Table~\ref{NS},
$$
v_{g}=a_{t_1}+a_{r_1}+\tfrac{1}{3}(a_{r_1g}+a_{gt_1})-\tfrac{64}{3}a_{t_1}a_{r_1}
$$
and 
$$
w_f=-\tfrac{1}{2^7}(3a_{t_2}+3a_{r_2}-a_{t_2f^2}-a_{t_2f^3}-a_{t_2f^4})+a_{t_2}a_{r_2}.
$$
By the linearity of the inner product and the previous lemmas, the computation of $(v_{g}, w_f)_V$ reduces to that of $(a_{t_1}a_{r_1}, a_{t_2}a_{r_2})_V$. Since $|t_1t_2|=2$, by Lemma~\ref{sh}, $\langle \langle a_{t_1}, a_{t_2}\rangle \rangle$ is of type $2B$, whence, by Table~\ref{NS},  $a_{t_1} a_{t_2}=0$ and so, by Lemma~\ref{ob}, 
\begin{equation}\label{eqterza}
(a_{t_1}a_{r_1}, a_{t_2}a_{r_2})_V=(a_{r_1}a_{t_2},a_{t_1}a_{r_2})_V.
\end{equation}
Since $|t_1r_2|=|t_2r_1|=3$, the algebras $\langle \langle a_{t_1},a_{r_2}\rangle \rangle$ and $\langle \langle a_{t_2},a_{r_1}\rangle \rangle$ are of type $3A$. Using  Table~\ref{NS}, we can express the vectors $a_{t_1}a_{r_2}$ and $a_{t_2}a_{r_1}$ as linear combinations of axes and a $3$-axis. Substituting $a_{t_1}a_{r_2}$ and $a_{t_2}a_{r_1}$ with these expressions, by Lemmas~\ref{23}, and~\ref{33}, we obtain 
$$
(a_{r_1}a_{t_2},a_{t_1}a_{r_2})_V= -\tfrac{1}{2^{17}}
$$
 and get 
 $$
 (v_g,w_f)_V=-\tfrac{7}{2^{11}}.
 $$  
 Finally since $|gf^2|=5$ and $|gf^{-2}|=2$, we obtain the value of $(v_g,w_f)_V$ in the last row as we did for the second row.
\end{proof}

\begin{lemma}\label{55}
Let  $f_1, f_2\in A_6^{(5)}$. The values of the inner products $(w_{f_1}, w_{f_2})_V$ are given in 
Table~\ref{tau55}.\footnote{Recall that, by definition, the set $A_6^{(5)}$ depends on the choice of a conjugacy class of $5$-elements in $A_6$. The possible values of $|f_1f_2|$ and $|f_1f_2^{-1}|$ depend on this choice. Numbers in parenthesis refer to the other conjugacy class.}
\end{lemma}
\begin{table}
$$
\begin{array}{|c|c|c|}
\hline
 |f_1f_2| & |f_1f_2^{-1}|  & (w_{f_1}, w_{f_2})\\
\hline
\hline
 5& 1 & \tfrac{5^37}{2^{19}}\\
\hline
 3 \:(5)&5 \:(3)& -\tfrac{21}{2^{19}}\\
\hline
 4 \:(4) & 4 \:(4) & \tfrac{133}{2^{19}}\\
\hline 
 5\: (2) & 5 \:(4) & \tfrac{119}{2^{20}}\\
\hline 

\end{array}
$$

\caption{The inner products between two $5$-axes}\label{tau55}
\end{table}
\begin{proof}
If $f_1$ and $f_2$ are as in the first row of Table~\ref{tau55}, then $f_1=f_2$, whence, by Table~\ref{NS}, $(w_{f_1}, w_{f_1})_V=\tfrac{5^37}{2^{19}}$.

Suppose $|f_1f_2|=3$ (resp. $|f_1f_2|=4$) and $|f_1f_2^{-1}|=5$ (resp. $|f_1f_2^{-1}|=4$). Then we may assume $f_1=(1,3,5,2,4)$ and $f_2=(2,3,6,4,5)$ (resp. $f_2=(2,3,4,5,6)$). Thus $(2,5)(3,4)$ inverts $f_i$ for $i\in \{1,2\}$ and the result for the second  (resp. third) row of Table~\ref{tau55} follows by Remark~\ref{start}. 

Finally suppose $|f_1f_2|=|f_1f_2^{-1}|=5$. Then we may assume $f_1=(1,2,3,4,5)$ and $f_2=(2,3,5,4,6)$. Set $t_1:=(2,5)(3,4)$, $r_1:=(1,2)(3,5)$, $t_2:=(2,4)(3,5)$, and $r_2:=(2,6)(3,4)$ so that $f_1=t_1r_1$ and $f_2=t_2r_2$. Then, by the Norton-Sakuma theorem, the linearity of the inner product and lemmas above, the computation of $(w_{f_1}, w_{f_2})_V$ reduces to that of $(a_{t_1}a_{r_1}, a_{t_2}a_{r_2})_V$. Now, $\langle \langle a_{t_1}, a_{t_2}\rangle \rangle$ is of type $2B$ and so, by Lemma~\ref{ob} we have
$$
(a_{t_1}a_{r_1}, a_{t_2}a_{r_2})_V=(a_{r_1}a_{t_2},a_{t_1}a_{r_2})_V.
$$
Since $|r_1t_2|=|t_1r_2|=3$, the algebras $\langle \langle a_{r_1}, a_{t_2}\rangle \rangle$ and $\langle \langle a_{t_1}, a_{r_2}\rangle \rangle$ are both of type $3A$, hence the products $a_{r_1}a_{t_2}$ and $a_{t_1}a_{r_2}$ are linear combinations of axes and $3$-axes. Thus, by the same argument used in Lemma~\ref{45} to compute the last row of Table~\ref{tau45}, we get $(a_{r_1}a_{t_2},a_{t_1}a_{r_2})_V=\tfrac{3\cdot 19}{2^{18}}$, and the result follows.
\end{proof}

\begin{proposition}\label{dime}
$\:$
\begin{enumerate}
 \item $\dim(V^{(2A)})=45$;
  \item $\dim(V^{(3A)})=40$;
   \item $\dim(V^{(4A)})=45$.
\end{enumerate}
\end{proposition}
\begin{proof}
 This follows by computing the ranks of the Gram matrices associated to the inner product with respect to the sets $\mathcal A$, $3Y$, and $4Y$, whose entries are given in Tables~\ref{NS}-\ref{tau55} (see~\cite[r2xA6\_innerproduct.g]{code}). 
\end{proof}

\section{Maximal subalgebras}\label{maxsub}

Since the maximal subgroups of $A_6$ isomorphic to $A_5$ or $S_4$ are generated by involutions  in $T$, the axes $a_{t}$ corresponding to these involutions  generate  Majorana subalgebras of $V$.
We collect here some results about these subalgebras. 

Let $S$ be a subgroup of $A_6$ isomorphic to $S_4$.
By Proposition~\ref{maximal}.$(1)$, all nine involutions in $S$ are Miyamoto involutions, hence
$V(S)$ affords a Majorana representation of $S$ with nine Majorana axes. By Lemma~\ref{sh} and condition (\ref{hyp}), this representation has shape $(2B, 3A, 4A)$, hence, by~\cite[Table~4]{MS}, it is unique and has dimension $25$.
Denote, as usual, by $S^\prime$ the derived subgroup of $S$ and by $\hat V(S)$ the subalgebra of $V(S)$ generated by the set $\{a_t\:|\:t\in S^{(2)}\setminus S^\prime\}$. Then $\hat V(S)$ affords the Majorana representation of $S$ with six Majorana axes and shape $(2B, 3A)$. By~\cite[Section 5]{IPSS}, $\hat V(S)$ has dimension $13$, it is not $2$-closed, and a basis for $\hat V(S)$ is given by the set 
$$
\{a_{t} \,|\,  t\in S^{(2)}\setminus S^\prime \}\cup \{u_h \,|\, h \in S^{(3)}\}\cup  \{
\overline v_{g}, \,|\, g \in S^{(4)}\}.
$$ 
Observe that, if $g\in A_6^{(4)}$, say $g=(1,2)(3,4,5,6)$, then $g$ is contained in exactly two subgroups of $A_6$ isomorphic to $S_4$, namely 
$$
S:=\langle (1,2)(5,6), (3,4,5) \rangle \:\mbox{ and }\:
S^\ast:=\langle (1,5)(2,3), (1,3,4)(2,5,6) \rangle .
$$
It follows that there are two distinct subalgebras $\hat V(S)$ and $\hat V(S^\ast)$ of $V$, each of which containing a \ghost $4$-axis corresponding to $\langle g\rangle$.
For the remainder of the paper we assume that condition (M8D) holds.
\medskip

In the next section we shall prove that every \ghost $4$-axis is contained in the $2$-closure of $V$. 
To this aim we shall need to compute inner products between axes and \ghost $4$-axes and to complete this task we shall use some information on algebra multiplication which we deduce from the maximal subalgebras of $V$. 
\medskip

Note that $|S^{(3)}|=4$  and $|S^{(4)}|=3$.  
Let $S^{(3)}=\{ h_1, h_2, h_3, h_4 \}$ and for an involution $s\in S^\prime$ define
\begin{equation}\label{Eqgamma}
\gamma_{s}:=a_{s}\cdot (u_{h_1}+u_{h_2}+u_{h_3}+u_{h_4}).
\end{equation} 
Similarly, let $S^{(4)}=\{g_1, g_2, g_3\}$ and for $s\in S^{(2)}\setminus S^\prime$ define
\begin{equation}\label{Eqeta}
\eta_{s}:=a_{s}\cdot (v_{g_1}+v_{g_2}+v_{g_3}).
\end{equation}
Since, every $x \in C_{A_6}(s)$ acts trivially on $v_{g_1}+v_{g_2}+v_{g_3}$ and on $u_{h_1}+u_{h_2}+u_{h_3}+u_{h_4}$, we have
\begin{lemma}\label{invariance}
$\eta_{s}$ and $\gamma_{s}$ depend only on $S$ and $s$.
\end{lemma}

\begin{proposition}\label{2xS4}
With the above notation, the following assertions hold
\begin{enumerate}
\item $V(S)$ has dimension $25$ and it is $3$-closed;
\item a basis of $V(S)$ is given by the union of the following sets
\begin{enumerate}
 \item [a)]  $\{a_{s} \,|\, s\in S^{(2)}\}$ ($2$-axes),
 \item [b)]  $\{u_h\,|\, h \in S^{(3)}\}$ ($3$-axes),
 \item [c)] $\{v_{g_1},v_{g_2}, v_{g_3}\}$ ($4$-axes),
  \item [d)]  $\{\overline v_{g_1},\, \overline v_{g_2},\,\overline v_{g_3}\}$ (\ghost $4$-axes),
\item [e)] $\{\eta_{s}\,|\,s\in S^{(2)}\setminus S^\prime\}$,
\end{enumerate}
\item for every $g\in S^{(4)}$, we have $a_{g^2}\cdot \ov_{g}=0$ and  $v_g\cdot \ov_{g}=0$,
\item for every  $ g\in S^{(4)}$, we have 
\begin{eqnarray}\label{formulagamma}
\gamma_{ g ^2}&=&\tfrac{4}{15}a_{g^2}+\tfrac{26}{135}\sum_{s\in S^\prime\setminus \langle  g^2\rangle} a_{s}-\tfrac{44}{135}\sum_{s\in C_S( g^2)^{(2)}\setminus S^\prime}a_{s}+\tfrac{1}{16}\sum_{h \in S^{(3)}} u_h \\
&& -\tfrac{2}{9} v_{ g}-\tfrac{2}{15}\sum_{l\in S^{(4)}\setminus \{ g\}} v_l+\tfrac{32}{45}\sum_{s\in C_S( g^2)^{(2)}\setminus S^\prime}\eta_{s}. \nonumber 
\end{eqnarray}
\end{enumerate}
\end{proposition}
\begin{proof}
Claims {\it (1)}, {\it (2)} and {\it (3)} have been proved in~\cite{Marta} and can be checked using~\cite[algebra2xS4.g]{code}.
The formula for $\gamma_{g ^2}$ in {\it (4)} has been computed using the construction of the algebra $V(S)$ in~\cite{Marta} (see~\cite[algebra2xS4.g]{code}).
\end{proof}

\begin{cor}\label{product44f}
For every $ g\in A_6^{(4)}$ we have $a_{g^2}\cdot \ov_{g}=0$ and $v_g\cdot \ov_{g}=0$.
\end{cor}
\begin{proof}
If $g\in A_6^{(4)}$, then $g$ is contained in maximal subgroup of $A_6$ isomorphic to $S_4$ and the result follows by Proposition~\ref{2xS4}.{\it (3)}.
\end{proof}

\begin{lemma}\label{formulaa2*v5}
With the notation above, let $g\in S^{(4)}$ and let $g^2, t,r$ be the involutions of $S^\prime$. Then 
\begin{eqnarray*}
a_{t}\cdot \overline v_g&=&\tfrac{5}{36}a_{ g^2}+\tfrac{1}{16}a_{t}+\tfrac{1}{48}a_{r}+\tfrac{1}{9}\sum_{s\in N_S( g)^{(2)}\setminus S^\prime} a_{s} -\tfrac{11}{72}\sum_{s\in S^{(2)}\setminus C_S( g^2)} a_{s}\\
&&+ \tfrac{15}{512}\sum_{h \in S^{(3)}} u_h-\tfrac{5}{48}\sum_{l \in S^{(4)}\setminus \{g\}}v_l+\tfrac{1}{24} v_{g}\\
&& -\tfrac{1}{3}\left (\sum_{r\in N_S(g)^{(2)}\setminus S^\prime} \eta_{r} -\sum_{r\in S^{(2)}\setminus C_S(g^2)} \eta_{r} \right )
\end{eqnarray*}
\end{lemma}
\begin{proof}
This formula has been computed in~\cite[algebra2xS4.g]{code}.
\end{proof}

Let
\begin{equation}\label{defU}
U:=V^{(2A)}+ V^{(3A)}+ V^{(4A)}
\:\: \mbox{ and } \:\:U^{\bullet}:=U+ \langle \overline v_{l} \mid l\in G^{(4)} \rangle .
\end{equation} 

Since $\Miy(\A)=A_6$, $U$ and  $U^{\bullet}$ are invariant under the action of $\Miy(\A)$.

\begin{lemma}\label{w4}
With the above notation, let $s\in S^{(2)}\setminus S^\prime$ and let $r$, $t$, and $rt$ be the involutions in $S^\prime$. Then there are elements $h$ and $h_1$ in $S^{(3)}$, $s\in S^{(2)}\setminus S^\prime$  
and a unique element $g_1\in S^{(4)}$, such that
\begin{enumerate}
\item   $rs=sr$,  
\item  $h^{s}=h^{-1}$ and $h_1^{s}=h_1^{-1}$,  
\item  $rg_1=g_1r$, 
\end{enumerate}
 moreover 
\begin{eqnarray}\label{formica}
a_{r}\cdot u_h&=&\tfrac{1}{15}a_{r}+\tfrac{13}{270}(a_{t}+a_{rt}) -\tfrac{49}{270}a_{s}+\tfrac{1}{54}a_{rs}+\tfrac{3}{64}u_h+\tfrac{1}{64}u_{h_1} \nonumber\\
&&-\tfrac{1}{30}\sum_{g\in S^{(4)}\setminus \{g_1\} } v_g -\tfrac{1}{18}v_{g_1}+\tfrac{16}{45}\eta_{s} .
  \end{eqnarray}
  In particular, $\eta_s\in a_{r}\cdot u_h+U$.
\end{lemma} 
\begin{proof}
The first part follows by a direct inspection in the group $S\cong S_4$, the last  formula has been computed in~\cite[algebra2xS4.g]{code}. 
\end{proof}

\begin{cor}\label{etas}
Let $s\in A_6^{(2)}$. Then there exist $r\in A_6^{(2)}$ and $h\in A_6^{(3)}$ such that $\eta_s\in a_{r}\cdot u_h+U$.
  \end{cor}
  \begin{proof}
This follows immediately by Equation~(\ref{formica}).
        \end{proof}

\begin{proposition}\label{2eta}
Let $ g\in S^{(4)}$. Then 
$ \gamma_{g^2}\in U^\bullet .
$
\end{proposition}
\begin{proof}
By Proposition~\ref{2xS4}.{\it (4)} it is enough to prove that  
\begin{equation}\label{2etaeq}
  \sum_{r\in C_S( g^2)^{(2)}\setminus S^\prime}\eta_{r}\in U^\bullet.  
\end{equation}
We may assume w.l.o.g. that $S=\langle  (1,2)(5,6), (3,4,5) \rangle$ and $g=(1,2)(3,4,5,6)$, so that
$$C_S( g^2)^{(2)}\setminus S^\prime=\{(1,2)(3,5), (1,2)(4,6)\}.$$ Let $S^\ast:=\langle(1,5)(2,3), (1,3,4)(2,5,6)\rangle$. Then, as mentioned at the beginning of this section, $S$ and $S^\ast$
are the two subgroups of $A_6$ isomorphic to $S_4$ that contain $g$. Since  $S\cap S^\ast=N_{A_6}(\langle g\rangle)\cong D_8$, the Klein subgroups of $S$ and $S^\ast$ are both contained in $S\cap S^\ast$ and they intersect exactly in $\langle g^2\rangle$.  Let $t:=(3,6)(4,5)$. Then $t$ belongs to the Klein subgroup $K$ of $S$, but $t$ does not belong to the Klein subgroup of $S^\ast$. Let $\hat V(S^{\ast })$ be the subalgebra of $V(S^\ast)$ generated by the axes $a_t$ associated to the involutions $t$ of $S^\ast$ that are not contained in its derived subgroup. By the discussion at the beginning of this section, $\hat V(S^{\ast })$ affords a Majorana representation of $S_4$ of type $(2B, 3A)$. Hence, by the formula in~\cite[p.2462]{IPSS}, the product $a_{t}\cdot \overline v_{ g}$ belongs to $U^\bullet$.  Thus, by Lemma~\ref{formulaa2*v5}, we get 
$$
\sum_{r\in N_S(g)^{(2)}\setminus S^\prime} \eta_{r} -\sum_{r\in S^{(2)}\setminus C_S( g^2)} \eta_{r}  \in U^\bullet
$$
that is
\begin{eqnarray}\label{pr1}
&\eta_{(1,2)(3,5)}+\eta_{(1,2)(4,6)}-\eta_{(1,2)(3,4)}-\eta_{(1,2)(3,6)}-\eta_{(1,2)(4,5)}-\eta_{(1,2)(5,6)}\in U^\bullet .
\end{eqnarray}
Applying  $(1,2)(3,4)$ to Equation~(\ref{pr1}) we get 
\begin{equation}\label{pr2}
-\eta_{(1,2)(3,5)}-\eta_{(1,2)(4,6)}-\eta_{(1,2)(3,4)}+\eta_{(1,2)(3,6)}+\eta_{(1,2)(4,5)}-\eta_{(1,2)(5,6)}\in U^\bullet
\end{equation}
whence, taking the sum and dividing by $-2$, we get
$$
\eta_{(1,2)(3,4)}+\eta_{(1,2)(5,6)}\in U^\bullet
$$
which is the condition in Equation~(\ref{2etaeq}). 
\end{proof}

We consider now the subalgebras of $V$ corresponding to the maximal subgroups of $A_6$ isomorphic to $A_5$. So let $L$ be a subgroup of $A_6$ with $L\cong A_5$ and let $V(L)$ be the subalgebra of $V$ generated by the axes $a_t$ with $t\in T\cap L$. By Lemma~\ref{sh}, it follows that $V(L)$ affords a Majorana representation of $A_5$ with shape $(2B, 3A, 5A)$. 

\begin{lemma}\label{A5}
The alternating group $A_5$ has a unique Majorana representation $W$ of shape $(2B, 3A, 5A)$. $W$ has dimension $46$ and a basis for $W$ is given by the set consisting of all $2$-axes, $3$-axes, $5$-axes and the vectors $\gamma_{(i,j)(k,l)}$ as defined in Equation~(\ref{Eqgamma}), for every $(i,j)(k,l)\in A_5$.
\end{lemma}
\begin{proof}
By~\cite[Table 4]{MS} there is a  unique Majorana representation of $A_5$ with shape $(2B, 3A, 5A)$ and this representation has dimension $46$. Using the Gap package ``MajoranaAlgebras"~\cite{MM} one sees that  the set consisting of all $2$-axes, $3$-axes, $5$-axes, and the vectors $\gamma_{(i,j)(k,l)}$ has size $46$ and its elements are linearly independent (see~\cite[repA5(2B).g]{code}). 
\end{proof}

\begin{lemma}\label{2-3inA5}
Let $L$ be a subgroup of $A_6$ isomorphic to $A_5$.
Let $t\in L^{(2)}$ and $h\in L^{(3)}$ be such that $\langle t, h\rangle\cong A_4$.  Let $r\in C_L(t)\setminus \{1, t\}$ and let 
\begin{enumerate}
    \item $L^{(3)}_2(h)\cap L^{(3)}_2(t)=\{h_1\}$, \item $L^{(3)}_5(h)\cap L^{(3)}_2(t)=\{h_2\}$, 
    \item $L^{(5)}_2(h)\cap L^{(5)}_2(t)=\{f\}$, 
\end{enumerate}
In the algebra $V(L)$, we have
\begin{eqnarray*}
a_{t}\cdot u_{h}&=&\tfrac{1}{32}u_{h}+ 
\tfrac{1}{45}\sum_{s\in L^{(2)}_3(t)\cap L^{(2)}_5(h)} (a_{s} -a_{s^r}) \\
&&+\tfrac{1}{36}\sum_{s\in L^{(2)}_5(t)\cap L^{(2)}_3(h)}\left (a_{s}+a_{s^t}-a_{s^r}-a_{s^{tr}}\right )\\
&& +\tfrac{1}{64} \left (u_{h_1} -u_{h_2}-\sum_{k\in L^{(3)}_3(t)\setminus \{h,h^t, (h^{-1})^t\}} u_k\right ) \\ 
&&-\tfrac{32}{45}\left (w_{f}-w_{f^r}\right )\\
&&  +\tfrac{1}{4}\gamma_t -\tfrac{1}{8}\sum_{s\in L^{(2)}_3(t)\cap L^{(2)}_5(h)} (\gamma_s -\gamma_{s^r}) \\
&&-\tfrac{1}{8}\sum_{s\in L^{(2)}_5(t)\cap L^{(2)}_3(h)}\left (\gamma_{s}+\gamma_{s^t}-\gamma_{s^r}-\gamma_{s^{tr}}\right )
  \end{eqnarray*}
In particular, $a_{t}\cdot u_h\in (w_{f}-w_{f^r})+U^\bullet$.
\end{lemma}
\begin{proof}
The formula has been computed in GAP using the package ``MajoranaAlgebras" (see~\cite[repA5(2B).g]{code}). 
\end{proof}

\begin{lemma}\label{w4fine}
For every $s\in A_6^{(2)}$, there are $f_1, f_2 \in A_6^{(5)}$ such that 
$$\eta_{s}\in (w_{f_1}-w_{f_2})+U^\bullet .
$$
\end{lemma}
\begin{proof}
By Corollary~\ref{etas}, there exist $t\in A_6^{(2)}$ and $h\in A_6^{(3)}$ such that $\langle t, h\rangle \cong A_4$ and $\eta_{s}\in a_{t}\cdot u_h +U^\bullet$. Let $L$ be a maximal subgroup of $A_6$ isomorphic to $A_5$ and containing $\langle t, h\rangle$. 
By Lemma~\ref{2-3inA5}, there exist $f_1, f_2 \in L^{(5)}$ such that $a_{t}\cdot u_h\in (w_{f_1}-w_{f_2})+U^\bullet$ and the result follows. 
\end{proof}

\section{More inner products}\label{moreinner}

In this section we show that $U^{\bullet}=U$. Since by definition $U \leq U^{\bullet}$, it is enough to show that the two subspaces have the same dimension. To do that, we first compute the inner products between $N$-axes, $N\in \{2,3,4,5\}$,  and \ghost $4$-axes. For the remainder of this paper, we let $$G:=A_6.$$

\begin{lemma}\label{innerf1}
Let $t\in G^{(2)}$, $h\in G^{(3)}$, $g, l\in G^{(4)}$. The following holds
\begin{enumerate}
\item if $t=l^2$, then $(a_{t}, \ov_{l})_V=0$;
\item if $\langle t, l \rangle \cong D_8$, then $ (a_{t}, \ov_{l})_V=\tfrac{1}{24}$;
\item if $\langle t, l\rangle \cong S_4$, then $ (a_{t}, \ov_{l})_V=\tfrac{31}{2^6\cdot 3}$;
\item if $\langle h, l\rangle \cong S_4$, then $ (u_{h}, \ov_{l})_V=\tfrac{11}{27}$;
\item  $ (v_{g}, \ov_{g})_V=0$;
\item if $\langle g, l\rangle \cong S_4$, then $ (v_{g}, \ov_{l})_V=\tfrac{11}{48}$ and $ (\ov_{g}, \ov_{l})_V=\tfrac{9}{16}$;
\item $ (\ov_{l}, \ov_{l})_V=2$;
\item $(\eta_t, \eta_t)_V=\tfrac{4565}{3\cdot 2^{12}}$.
\end{enumerate}
\end{lemma}
\begin{proof}
These inner products can be obtained inside a Majorana subalgebra $V(S)$, with $S\cong S_4$, of dimension $25$ and shape $(2B,3A, 4A)$. Hence the result follows by~\cite[Section 4.1]{Marta}, where all these inner products have been computed. \end{proof}

\begin{lemma}\label{innerf2}
Let $t\in G^{(2)}$, $h\in G^{(3)}$, $g, l\in G^{(4)}$, and $f\in G^{(5)}$. The following holds
\begin{enumerate}
\item if $\langle t, l\rangle \cong 3^2:4$, then $ (a_{t}, \ov_{l})_V=\tfrac{27}{2^8}$;
\item if $\langle t, l\rangle \cong A_6$, then $ (a_{t}, \ov_{l})_V=\tfrac{31}{384}$;
\item if $h\in G^{(3)}_3(l)\cup G^{(3)}_5(l^{-1})$, then $ (u_{h}, \ov_{l})_V=\tfrac{2}{27}$;
\item if $\langle g, l\rangle \cong 3^2:4$, then $ (v_{g}, \ov_{l})_V=\tfrac{117}{256}$;
\item  if $g\in G^{(4)}_{5}(l)\cap G^{(4)}_5(l^{-1})$, then $ (v_{g}, \ov_{l})_V=\tfrac{5}{48}$;
\item if $\langle g, l\rangle \cong 3^2:4$, then $ (\ov_{g}, \ov_{l})_V=\tfrac{53}{256}$;
\item if $\langle h, l\rangle \cong 3^2:4$, then $ (u_{h}, \ov_{l})_V=\tfrac{3}{20}$;
\item if $h\in G^{(3)}_{5}(l)\cap G^{(3)}_{5}(l^{-1})$, then $ (u_{h}, \ov_{l})_V=\tfrac{1}{12}$;
\item if $g\in G^{(4)}_{5}(l)\cap G^{(4)}_5(l^{-1})$, then $ (\ov_{g}, \ov_{l})_V=\tfrac{31}{144}$;
\item  if $g\in G^{(4)}_{4}(l)\cup G^{(4)}_5(l^{-1})$, then $ (\ov_{g}, \ov_{l})_V=\tfrac{107}{1152}$.
\end{enumerate}
\end{lemma}
\begin{proof}
For $i\in \{1, \ldots , 10\}$, denote by $x_i$
 the inner product in claim {\it (i)}. 
We shall first express $x_7$, $x_8$, $x_9$, and $x_{10}$ in terms of $x_1, \ldots , x_6$. Then, from the fusion law, we shall derive linear relations involving $x_1, \ldots , x_6$ and determine the values of $x_1, \ldots ,x_6$. We refer to~\cite[algebraA6(2B).g]{code} for the explicit computations.

Let $h\in G^{(3)}$ be such that $\langle h, l\rangle \cong 3^2:4$, as in {\it (7)}. W.l.o.g. we may assume $l=(1,2)(3,6,4,5)$ and $h=(2,3,4)$. Then $h=l^2 s$ with $s=(2,3)(5,6)$, hence $\langle \langle a_{l^2}, a_{s}\rangle \rangle$ is a Norton-Sakuma algebra of type $3A$. 
Set 
\begin{equation}
\label{alpha}
\alpha:=u_h-\tfrac{8}{45}a_{l^2}-\tfrac{2^5}{45}(a_{s}+a_{l^2sl^2}).
\end{equation}
Then, by~\cite[Table~4]{IPSS}, 
$$
a_{l^2}\cdot \alpha=\tfrac{1}{4}\alpha,
$$ 
and, by Corollary~\ref{product44f},  
\begin{equation*}
    \label{tl}
a_{l^2}\cdot \ov_{l}=0.
\end{equation*} 
Whence (see Remark~\ref{start})
\begin{equation}
    \label{alpa}
(\alpha, \ov_{l})_V=0. 
\end{equation} 
Substituting in Equation~(\ref{alpa}) the expression of $\alpha$ given in Equation~(\ref{alpha}) and using Lemma~\ref{sh} and Table~\ref{NS} to obtain the values of the inner products between axes, we get
\begin{equation}
    \label{relx1}
x_7=(u_h, \ov_{l})_V=\tfrac{64}{45}(a_t, \ov_l)_V=\tfrac{64}{45}x_1.
\end{equation}

Let $h\in G^{(3)}_{5}(l)\cap G^{(3)}_{5}(l^{-1})$ as in {\it (8)}. We may assume $l=(1,2)(3,6,4,5)$ and $h=(1,2,5)$. Then $h=rs$ with $r=(2,5)(4,6)$ and $s=(1,2)(4,6)$. By ~\cite[Table~4]{IPSS} 
$$
\beta:=u_h-\tfrac{10}{27}a_{s}+\tfrac{2^5}{27}(a_{r}+a_{srs})
$$ 
and 
$$\hat \beta:=u_h-\tfrac{8}{45}a_{s}-\tfrac{2^5}{45}(a_{r}+a_{srs})
$$ are, respectively, a $0$-eigenvector and a $\tfrac{1}{4}$-eigenvector for $\ad_{a_{s}}$.
Since $\ov_{l}+\ov_{l}^s$ is invariant for $s$, by~\cite[Equations (3.5) and  (3.6)]{Rehren}, the vectors 
$$
\hat \delta:=a_{s}(\ov_{l}+\ov_{l}^s)-(\ov_{l}+\ov_{l}^s,a_{s})_V a_{s} 
$$ and 
$$
\delta:=\ov_{l}+\ov_{l}^s - (\ov_{l}+\ov_{l}^s,a_{s})_V a_{s}-4\hat \delta 
$$ 
are, respectively, a $\tfrac{1}{4}$-eigenvector and a $0$-eigenvector for $\ad_{a_{s}}$.
Hence $(\beta, \hat \delta)_V=0$ and $(\hat \beta, \delta)_V=0$. Since $\langle s, l\rangle \cong S_4$, using~\cite[Section~5]{IPSS}, we obtain expressions of 
the vectors $\delta$ and $\hat \delta$  as linear combinations of $N$-axes and \ghost $4$-axes in a basis of the subalgebra $V(\langle s,l\rangle)$. Thus, as in the previous case, a straightforward computation gives, by the  identities $(\hat \beta, \delta)_V=0$ and $(\beta, \hat \delta)_V=0$, respectively the relations  
\begin{equation}\label{relx1seconda}
x_8=(u_h, \ov_{l})_V=\tfrac{64}{81}x_1-\tfrac{256}{81}x_2+\tfrac{62}{243}  
\end{equation} 
and 
\begin{equation}\label{rel3}
x_2=\tfrac{1}{10}x_1+\tfrac{539}{7680}.
\end{equation}

Let $g\in G^{(4)}_{5}(l)\cap G^{(4)}_5(l^{-1})$ as in {\it (9)}. We may assume $g=(1,2)(3,5,6,4)$ and $l=(1,4,2,6)(3,5)$. Then, $l^2=(1,2)(4,6)=s$, and so, by Corollary~\ref{product44f}, $\ov_{l}$ is a $0$-eigenvector for $a_{s}$, whence $(\ov_{l}, \hat \delta)_V=0$. From this equality, we get, as in the previous cases,  
\begin{equation}
    \label{G5G5}
x_9=(\ov_{g}, \ov_{l})_V= \tfrac{4}{9}x_1+\tfrac{97}{576}.
\end{equation}

To compute the product $ (\ov_{g}, \ov_{l})_V$ when $g\in G^{(4)}_{4}(l)\cup G^{(4)}_5(l^{-1})$, as in {\it (10)}, we may assume $l=(1,2)(3,6,4,5)$ and $g=(1,4,5,6)(2,3)$. Let $s:=(1,2)(3,5)$. Then $$\langle s,g\rangle \cong S_4\cong \langle s,l\rangle.$$   
Proceeding as above, taking, in place of $\delta$ and $\hat \delta$, two orthogonal eigenvectors for $\ad_{a_{s}}$, one in the subalgebra $V(\langle s,g\rangle)$ and the other one in the subalgebra $V(\langle s,l\rangle)$, we get, using the structure of these subalgebras given in~\cite[Section 5]{IPSS},  
\begin{equation}
    \label{G4G4}
    x_{10}= (\ov_{g}, \ov_{l})_V=\tfrac{43}{420}x_1-\tfrac{9}{44}x_3-\tfrac{1}{10}x_6+\tfrac{961}{9216}.
\end{equation}

\medskip
We are now ready to compute $x_1,\ldots , x_6$. Let $s:=(1,2)(4,6)$ and
\begin{align*}
\hat \beta_1&:=v_{(1,2)(3,4,5,6)}-\tfrac{1}{3}a_{s}-\tfrac{2}{3}(a_{z(3,4)(5,6)}+a_{z(3,6)(4,5)})-\tfrac{1}{3}a_{z(1,2)(3,5)}, \\ 
\hat \beta_2&:=v_{(1,2,4,6)(3,5)}-\tfrac{1}{3}a_{s}-\tfrac{2}{3}(a_{z(1,4)(3,5)}+a_{z(2,6)(3,5)})-\tfrac{1}{3}a_{z(1,6)(2,4)},\\
\hat \delta_1&:=a_{s}(u_{(2,3,5)}+u_{(1,5,3)})-(u_{(2,3,5)}+u_{(1,5,3)},a_{s})_V a_{s},\\ 
\beta_1&:=v_{(1,2)(3,4,5,6)}-\tfrac{1}{2}a_{s}+2(a_{z(3,4)(5,6)}+a_{z(3,6)(4,5)}), \\ 
\beta_2&:=v_{(1,2,4,6)(3,5)}-\tfrac{1}{2}a_{s}+2(a_{z(1,4)(3,5)}+a_{z(2,6)(3,5)}),\\
\delta_1&:=u_{(2,3,5)}+u_{(1,5,3)}-(u_{(2,3,5)}+u_{(1,5,3)},a_{s})_V a_{s}-4\hat \delta_1. 
\end{align*}
By~\cite[Table~4]{IPSS} and, since $u_{(2,3,5)}+u_{(1,5,3)}$ is invariant for $s$, by~\cite[Equations (3.5) and  (3.6)]{Rehren}, $\hat \beta_1$, $\hat \beta_2$, and $\hat \delta_1$ are $\tfrac{1}{4}$-eigenvectors for $\ad_{a_{s}}$,
while $\beta_1$, $\beta_2$, and $\delta_1$
 are $0$-eigenvectors for $\ad_{a_{s}}$. Hence, by Remark~\ref{start}, we have
 $$
(\delta, \hat \delta_1)_V=(\hat \delta,\delta_1)_V= (\delta_1, \hat \beta_1)_V=(\hat \delta_1, \beta_1)_V=(\delta, \hat \beta_2)_V=0.
 $$
The identities  $(\delta, \hat \delta_1)_V=0$ and $(\hat \delta,\delta_1)_V=0$ imply respectively 
$$
\tfrac{3392}{6075}x_1-\tfrac{2}{9}x_3-\tfrac{16}{27}x_6+\tfrac{7801}{97200}=0
$$
 and 
 $$
 -\tfrac{2272}{6075}x_1+\tfrac{7}{9}x_3+\tfrac{112}{135}x_6-\tfrac{18461}{97200}=0,
 $$
whence $x_6= \tfrac{80}{63}x_1+\tfrac{131}{1792}$ and $x_3=-\tfrac{1376}{1575}x_1+\tfrac{6283}{37800}$. 
The identities  $(\delta_1, \hat \beta_1)_V=0$ and $(\hat \delta_1, \beta_1)_V=0$ imply $x_1=\tfrac{27}{256}$ and $x_5=\tfrac{5}{48}$, whence $x_2=\tfrac{31}{384}$, $x_3=\tfrac{2}{27}$ and $x_6=\tfrac{53}{256}$. Finally, the identity $(\delta, \hat \beta_2)_V=0$ implies $x_4=\tfrac{117}{256}$. 
\end{proof}

\pagebreak

 \begin{lemma}\label{innerf3}
 Let  $g, l\in G^{(4)}$ and $f\in G^{(5)}$. The following hold
\begin{enumerate}
\item if $f\in G^{(5)}_{5}(l)\cap G^{(5)}_5(l^{-1})$, then $ (w_{f}, \ov_{l})_V=0$;
\item if $f\in G^{(5)}_{3}(l)\cap G^{(5)}_5(l^{-1})$, then $ (w_{f}, \ov_{l})_V=\tfrac{175}{49152}$;
\item if $f\in G^{(5)}_{4}(l)\cap G^{(5)}_3(l^{-1})$, then $ (w_{f}, \ov_{l})_V=-\tfrac{175}{49152}$;
\item if $f\in G^{(5)}_{4}(l)\cap G^{(5)}_5(l^{-1})$, then $ (w_{f}, \ov_{l})_V=\tfrac{49}{6144}$;
\item if $f\in G^{(5)}_{2}(l)\cap G^{(5)}_5(l^{-1})$, then $ (w_{f}, \ov_{l})_V=-\tfrac{49}{6144}$;
\item  if $g\in G^{(4)}_{4}(l)\cup G^{(4)}_5(l^{-1})$, then $ (v_{g}, \ov_{l})_V=\tfrac{41}{384}$.
\end{enumerate}
\end{lemma}
\begin{proof}
Let $l$ and $f$ be as in {\it (1)}. We may assume $l=(1,2)(3,4,5,6)$ and $f=(2,4,5,3,6)$. By Corollary~\ref{product44f}, $\ov_{l}$ is a $0$-eigenvector for $\ad_{a_{l^2}}$. Let 
$$
\hat \beta:=w_f+\tfrac{1}{2^7}(a_{l^2f}+a_{l^2f^4}-a_{l^2f^2}-a_{l^2f^3}). 
$$ 
Since $\langle l^2, f\rangle\cong D_{10}$, $\langle \langle a_{l^2}, a_{l^2f}\rangle \rangle$ is a Norton-Sakuma algebra of type $5A$ and by~\cite[Table~4]{IPSS}, $\hat \beta$ is a $\tfrac{1}{4}$-eigenvector for $\ad_{a_{l^2}}$.
Thus $(\ov_{l}, \hat \beta)_V=0$. Since, by Lemmas~\ref{innerf1} and \ref{innerf2}, $(\ov_{l},a_{l^2f}+a_{l^2f^4}-a_{l^2f^2}-a_{l^2f^3})_V=0$, we get
$$(\ov_{l},w_f)_V=(\ov_{l},w_f+\tfrac{1}{2^7}(a_{l^2f}+a_{l^2f^4}-a_{l^2f^2}-a_{l^2f^3}))_V=(\ov_{l}, \hat \beta)_V=0.$$ 

Let $l$ and $f$ be as in {\it (2)}. We may assume $l=(1,2)(3,6,4,5)$ and $f=(1,4,3,5,6)$. Set $s:=(3,5)(4,6)$,
\begin{equation}
    \label{beta1}
\beta_1:=w_f+\tfrac{3}{2^9}a_{s}-\tfrac{15}{2^7}(a_{sf}+a_{sf^4})-\tfrac{1}{2^7}(a_{sf^2}+a_{sf^3})
\end{equation}
and 
$$
\hat \beta_1:=w_f+\tfrac{1}{2^7}(a_{sf}+a_{sf^4}-a_{sf^2}-a_{sf^3}). 
$$
Then $f^s=f^{-1}$ and so, by ~\cite[Table~4]{IPSS}, $\beta_1$ is a $0$-eigenvector and $\hat \beta_1$ is a  $\tfrac{1}{4}$-eigenvector for $\ad_{a_{s}}$, respectively.
Moreover, since $l^s=l^{-1}$, $\ov_{l}$ is $s$-invariant. By~\cite[Equations (3.5) and (3.6)]{Rehren}, it follows that the vectors
\begin{align*}
\hat \delta_1:=a_{s}\ov_{l}-(\ov_{l},a_{s})_V a_{s}\:\: \mbox{ and }\:\: 
 \delta_1:=\ov_{l}-(\ov_{l},a_{s})_V a_{s}-4\hat \delta_1 
\end{align*}
are, respectively, a $\tfrac{1}{4}$-eigenvector and a $0$-eigenvector for $\ad_{a_{s}}$, whence
$$
(\beta_1, \hat \delta_1)_V=(\delta_1, \hat \beta_1)_V=0.
$$
Now, substituting the expressions for $\beta_1, \hat \beta_1, \delta_1$, and $\hat \delta_1$ in   the identity $(\beta_1, \hat \delta_1)_V+(\delta_1, \hat \beta_1)_V=0$, by the results in Section~\ref{Scalar} and  Lemmas~\ref{innerf1} and \ref{innerf2}, we obtain $(\ov_{l}, w_f )_V=\tfrac{175}{49152}$. 

If $l$ and $f$ are as in {\it (3)}, then $l$ and $f^2$ satisfy the hypotheses of claim {\it (2)} and the result follows by Remark~\ref{remG}.

Let $l$ and $f$ be as in {\it (4)}. We may assume $l=(1,2)(3,6,4,5)$ and $f=(1,6,2,5,4)$. Set $s:=(1,2)(4,5)$ and define $\beta_1$ as in Equation~(\ref{beta1}). By ~\cite[Table~4]{IPSS}, $\beta_1$ is a $0$-eigenvector for $\ad_{a_{s}}$. Moreover, by~\cite[Equations (3.5) and (3.6)]{Rehren},
$$
\hat \delta_2:=a_{s}(\ov_{l}+\ov_{l^s})-(\ov_{l}+\ov_{l^s},a_{s})_V  a_{s}
$$
is a $\tfrac{1}{4}$-eigenvector for $\ad_{a_{s}}$ and it can be computed explicitly as a linear combination of axes and \ghost $4$-axes in the subalgebra $V(\langle s, l\rangle)$. From the identity $(\beta_1, \hat \delta_2)_V=0$ we get  $(\ov_{l}, w_f)_V=\tfrac{49}{6144}$. 

If $l$ and $f$ are as in {\it (5)}, then $l$ and $f^2$ satisfy the hypotheses of claim {\it (4)} and the result follows by Remark~\ref{remG}.

Finally, 
assume $g$ and $l$ are as in {\it (6)}. Fix $s\in G^{(2)}$,   then, by Lemma~\ref{innerf1}.(9), \begin{equation}\label{eqetas}
(\eta_s, \eta_s)_V=\tfrac{4565}{12288}.
\end{equation} 
By Lemma~\ref{w4fine}, $\eta_s$ has an expression as a linear combination of axes, odd axes and \ghost $4$-axes, which can be explicitly obtained using Proposition~\ref{2xS4}, Lemma~\ref{w4}, Proposition~\ref{2eta}, and Lemma~\ref{2-3inA5} (see~\cite[algebraA6(2B).g]{code}). 
Substituting this expression in Equation~\ref{eqetas}, by the linearity of the inner product and the values of the inner products  obtained so far, we get 
$$\tfrac{4565}{12288}=(\eta_s, \eta_s)_V=\tfrac{36643}{98304}-\tfrac{3}{256}(v_g, \ov_{l})_V,$$
whence
$$
(v_g, \ov_{l})_V=\tfrac{41}{384}.
$$
\end{proof} 

Let $g\in G^{(4)}$. Let $K_1$ and $K_2$ denote the two subgroups of $G$ isomorphic to $S_4$ containing $g$, such that, for $i\in \{1,2\}$, the elements of order $3$ in $K_i$ have cycle structure $3^i$. Further, let $Y$ be the set of all $2$-axes, $3$-axes and $4$-axes of $V$ (i.e. $Y=X\setminus 5Y$, in the notation used at the beginning of Section~\ref{Scalar}). Set
\begin{eqnarray*}
[\ov_{g}]_2&:=&  -\tfrac{200}{81}a_{g^2}-\tfrac{176}{81}\sum_{t\in K_1\cap G_2^{(2)}(g)}a_{t}-\tfrac{68}{81}\sum_{t\in K_2\cap G_2^{(2)}(g)}a_{t}+\tfrac{40}{81}\sum_{t\in G_2^{(5)}(g)}a_{t} \nonumber\\
&&-\tfrac{98}{81}\sum_{t\in K_1\cap G_2^{(3)}(g)}a_{t}+\tfrac{10}{81}\sum_{t\in K_2\cap G_2^{(3)}(g)}a_{t}\nonumber \\
&&-\tfrac{2}{81}\sum_{t\in G_2^{(4)}(g), [g,t]\in K_2\setminus \{1_G\}
}a_{t} +\tfrac{52}{81}\sum_{t\in G_2^{(4)}(g), [g,t] \in K_1\setminus \{1_G\}
}a_{t}; \nonumber 
\end{eqnarray*}

\begin{eqnarray*}
[\ov_{g}]_3 &=& -\tfrac{15}{16}\sum_{h\in (1,2,3)^G, |[g,h]|=5 }u_{h}
+\tfrac{25}{24}\sum_{h\in K_1^{(3)}}u_{h}+\tfrac{5}{6}\sum_{h\in (1,2,3)^G, |[g,h]|=2}u_{h} \\ 
&&+\tfrac{5}{48}\sum_{h\in K_2^{(3)} }u_{h}-\tfrac{5}{48}\sum_{h\in (1,2,3)(4,5,6)^G, |[g,h]|=2 }u_{h};
\end{eqnarray*}

\begin{eqnarray*}
[\ov_{g}]_4 &=&   -\tfrac{14}{27}\sum_{h\in G^{(4)}, |[g,h]|=5 }v_{h}+\tfrac{40}{27}\sum_{h\in G^{(4)}_2(g^2), |[g,h]|=3 }v_{h}  \nonumber \\
&&+\tfrac{10}{27}\sum_{h\in G^{(4)}, |[g,h]|=2 }v_{h}-\tfrac{2}{27}\sum_{h\in G^{(4)}_4(g^2), |[g,h]|=3 }v_{h}. \nonumber \\ 
\end{eqnarray*}
Note that, for $N\in \{2,3,4\}$, $[\ov_{g}]_N\in V^{(NA)}\leq \langle Y\rangle$.
 
\begin{lemma}
\label{L64} 
For every $g\in \G^{(4)}$,   $\ov_{g}$ admits the following expression as a linear combination of elements of $Y$:
\begin{eqnarray}\label{ov}
\ov_{g}&=& [\ov_{g}]_2+[\ov_{g}]_3+[\ov_{g}]_4. 
\end{eqnarray}
\end{lemma}

\begin{proof}
The inner products computed in this section and in Section~\ref{Scalar} imply  
$$
(\ov_{g}-[\ov_{g}]_2-[\ov_{g}]_3-[\ov_{g}]_4\:,\: \ov_{g}-[\ov_{g}]_2-[\ov_{g}]_3-[\ov_{g}]_4)_V=0,
$$
whence $\ov_{g}=[\ov_{g}]_2+[\ov_{g}]_3+[\ov_{g}]_4$ by the positive definiteness of the inner product. (for the explicit computation see~\cite[algebraA6(2B).g]{code}). 
\end{proof} 

\begin{cor} \label{4f}
 $U^\bullet= U$. 
\end{cor}
\begin{proof}
By Lemma~\ref{L64} $\ov_{g}\in U$. Now $U$ is invariant under the action of $G$, so the claim  follows because, by Hypothesis (M8D),  $\ov_{g}$ depends uniquely on $\langle g \rangle$ and  $G$ is transitive on its subgroups of order $4$.  
\end{proof}


\section{$5$-axes}\label{5assi}

Let 
\begin{equation}\label{somma5}
w:=\sum_{e\in G^{(5)}}w_e
\end{equation}
be the sum of all the $36$ positive $5$-axes of $V$.
Moreover, for $f\in G^{(5)}$ set
\begin{eqnarray}\label{wf2}
 [w_f]_2&:=& \tfrac{5}{768}\left (\sum_{t\in G_3^{(2)}(f)}a_{t}-\sum_{t\in G_5^{(2)}(f), \langle f, t\rangle \cong A_5}a_{t}\right ) \nonumber\\
&&
+\tfrac{5}{192}\left (\sum_{t\in G_5^{(2)}(f), \langle f, t\rangle =G}a_{t}-\sum_{t\in G_4^{(2)}(f)}a_{t}\right );
\end{eqnarray}
\begin{eqnarray}\label{wf3}
 [w_f]_3&:=& \tfrac{165}{16384}\left(  \sum_{h\in G^{(3)},  [f,h]\in f^G}
u_h-\sum_{h\in G^{(3)}, |[f,h]|=3} u_h\right )\nonumber \\
&& +\tfrac{105}{16384}\left (\sum_{h\in G^{(3)},  |[f,h]|=5, [f,h]\not \in f^G}
u_h - \sum_{h\in G^{(3)}, |[f,h]|=2} u_h   \right );  
\end{eqnarray}
\begin{eqnarray}\label{wf4}
 [w_f]_4&:=& \tfrac{1}{64}\left ( \sum_{h\in G^{(4)}, |[f,h]|=2} v_h +\sum_{h\in G^{(4)}, [f,h]\in f^G}
v_h\right ) \nonumber \\
&& -\tfrac{1}{64}\left ( \sum_{h\in G^{(4)}, |[f,h]|=4, (w_f, v_h)_V\neq 0} v_h +\sum_{h\in G^{(4)}, |[f,h]|=5, [f,h]\not \in f^G}v_h\right ).  
\end{eqnarray}
\begin{lemma}\label{rel5assi}
For every $f\in G^{(5)}$, $w_f$ admits the following expression as a linear combination of $w$ and elements of $Y$:
\begin{eqnarray} \label{w}
w_f&=&\tfrac{1}{36}w +[w_f]_2+[w_f]_3+[w_f]_4
\end{eqnarray}
\end{lemma}
\begin{proof}
 Using the values of the inner products between axes obtained in Sections~\ref{Scalar} and~\ref{moreinner},  we get  
 $$
 (w_f-\tfrac{1}{36}w -[w_f]_2-[w_f]_3-[w_f]_4\:,\: w_f-\tfrac{1}{36}w -[w_f]_2-[w_f]_3-[w_f]_4)_V=0
 $$ 
 (see~\cite[algebraA6(2B).g]{code}), whence $w_f=\tfrac{1}{36}w +[w_f]_2+[w_f]_3+[w_f]_4$, by the positive definiteness of the inner product. 
\end{proof}

Now let $K$ be a subgroup of $G$ isomorphic to $S_4$ and let $s$ be an element of order $2$ in $ K\setminus K^\prime$. By Proposition~\ref{maximal}, there are $r \in K^\prime$ of order two and $g\in G^{(4)}$ such that $C_K(s)=\{1_G, s, r,  g^2\}$.  Let $\bar t\in K^{(2)}$ be  such that the Klein subgroup of $K$ is  $\{1_G, r, \bar t,\bar t^s\}$. 
Then, set
\begin{eqnarray}\label{etas2}
[\eta_s]_2&:=&\tfrac{15}{32}a_{s}-\tfrac{1}{24}a_{r}-\tfrac{1}{36}(a_{\bar t}+a_{\bar t^s})-\tfrac{23}{288}a_{g^2}+\tfrac{11}{144}\sum_{t\in C_G(s)^{(2)}\setminus C_G(r)} a_{t} \nonumber \\
 &&  +\tfrac{1}{8}\sum_{t\in K^{(2)}\setminus \{s,r, \bar t,\bar t^s, g^2\}} a_{t}-\tfrac{17}{144}\sum_{t\in C_G(g^2)^{(2)}\setminus C_G(s)} a_{t} \nonumber \\
&& +\tfrac{11}{144} \sum_{t\in C_G(s)^{(2)}\setminus K} a_{t}
+\tfrac{1}{16}\sum_{\textsuperscript{$\begin{array}{c} t\in G^{(2)}, \\
|ts|=4, |tr|=3\end{array}$}}a_{t}\\
&& +\tfrac{1}{72}\sum_{t\in G^{(2)}, |tg^2|=5} a_{t} +\tfrac{1}{36}\sum_{\textsuperscript{$\begin{array}{c}t\in G^{(2)},\\ |ts|=5, |tr|=3\end{array}$}} a_{t} ;   \nonumber  
\end{eqnarray}

\begin{eqnarray}\label{etas3}
 [\eta_s]_3&:=&  -\tfrac{15}{128}\sum_{\textsuperscript{$ \begin{array}{c}h\in G^{(3)},\\ |\langle h,s\rangle|\leq 24, \langle h,g^2\rangle\cong S_3\end{array}$}} u_h +\tfrac{5}{64}\sum_{\textsuperscript{$ \begin{array}{c}h\in G^{(3)},\\ \langle h,s\rangle\cong A_4, \langle h,g^2\rangle\cong S_4\end{array}$}} u_h \nonumber\\ 
&& +\tfrac{5}{128}\sum_{\textsuperscript{$ \begin{array}{c}h\in G^{(3)},\\ \langle h,s\rangle\cong S_4, \langle h,g^2\rangle\cong A_4\end{array}$}} u_h +\tfrac{5}{128}\sum_{\textsuperscript{$ \begin{array}{c}h\in G^{(3)},\\ \langle h,s\rangle\cong S_3, \langle h,g^2\rangle\cong A_4\end{array}$}} u_h\nonumber \\ 
&& +\tfrac{5}{512}\sum_{\textsuperscript{$ \begin{array}{c}h\in G^{(3)},\\ \langle h,s\rangle\cong \langle h,g^2\rangle\cong S_4\end{array}$}} u_{h}- \tfrac{15}{512}\sum_{\textsuperscript{$ \begin{array}{c}h\in G^{(3)}\cap (1,2,3)^G,\\ \langle h,s\rangle\cong \langle h,g^2\rangle\cong A_5\end{array}$}} u_{h} \\
&& -\tfrac{5}{512}\sum_{\textsuperscript{$ \begin{array}{c}h\in G^{(3)},\\ \langle h,s\rangle\cong S_3, \langle h,g^2\rangle\cong S_4\end{array}$}} u_{h}- \tfrac{5}{512}\sum_{\textsuperscript{$ \begin{array}{c}h\in G^{(3)},\\ \langle h,s\rangle\cong \langle h,g^2\rangle\cong A_4\end{array}$}} u_{h}; \nonumber  
\end{eqnarray}

\begin{eqnarray}\label{etas4}
 [\eta_s]_4&:=& \tfrac{1}{24} v_g  -\tfrac{1}{48}\sum_{\textsuperscript{$ \begin{array}{c}l\in G^{(4)},\\
l^s=l^{-1}, \langle l,g^2\rangle \cong S_4 \end{array}$}} v_l-\tfrac{1}{48}\sum_{\textsuperscript{$ \begin{array}{c}l\in K^{(4)},\\
 \langle l, s\rangle =K\end{array}$}} v_l \nonumber \\
 && -\tfrac{1}{16}\sum_{\textsuperscript{$ \begin{array}{c}l\in G^{(4)},\\
 |\langle l, g^2\rangle |=36, |\langle l, s\rangle |\in \{24, 36\}\end{array}$}} v_l-\tfrac{1}{48}\sum_{\textsuperscript{$ \begin{array}{c}l\in G^{(4)},\\
 \langle l, g^2\rangle =G\end{array}$}} v_l \nonumber \\
&& +\tfrac{1}{48}\sum_{\textsuperscript{$ \begin{array}{c}l\in G^{(4)},\\
 \langle l, s\rangle =G, |\langle l, g^2\rangle |=36\end{array}$}} v_l+\tfrac{5}{48}\sum_{\textsuperscript{$ \begin{array}{c}l\in G^{(4)},\\
  |\langle l, s\rangle |=36, \langle l,g^2\rangle \cong S_4\end{array}$}} v_l\\ 
  && +\tfrac{7}{48}\sum_{\textsuperscript{$ \begin{array}{c}l\in G^{(4)},\\
l^{g^2}=l^{-1}, \langle l, s\rangle \cong S_4\end{array}$}} v_l+\tfrac{5}{32}\sum_{\textsuperscript{$ \begin{array}{c}l\in G^{(4)},\\
l^s=l^{-1}, 
l^{g^2}=l^{-1}\end{array}$}} v_l .\nonumber   \end{eqnarray}

\begin{cor}\label{coreta}
With the above notation, the following hold
\begin{enumerate}
\item $V^\circ=\langle w\rangle +U$ and $\dim (V^\circ)=\dim(U)+1=121$;
\item if $f_1, f_2\in G^{(5)}$, then $w_{f_1}-w_{f_2}\in U$; 
\item if $S$ is a subgroup of $G$ isomorphic to $S_4$ and  $s\in S^{(2)}\setminus S^\prime$, then $\eta_s\in U$ and 
admits the following expression as a linear combination of elements of $Y$:
\begin{eqnarray}\label{eqneta}
\eta_{s}&=&[\eta_s]_2+[\eta_s]_3+[\eta_s]_4.
\end{eqnarray}
In particular, $V(S)\leq U$.
\item if $H$ is a subgroup of $G$ isomorphic to $A_5$, then $V(H)\leq V^\circ$. 
\end{enumerate}
\end{cor}
\begin{proof}
Since $V^\circ=U+V^{(5A)}$, by Lemma~\ref{rel5assi} we have $V^\circ=\langle w\rangle +U$. A direct computation in GAP of the Gram matrices $\Gamma$ and $\Gamma^\ast$ associated to $Y$ and $Y\cup \{w\}$, respectively, gives  {\it (1)} (see~\cite[algebraA6(2B).g]{code}).
The second assertion follows immediately by Equation~(\ref{w}).

Now let $S$ be a subgroup of $G$ isomorphic to $S_4$ and let $s\in S^{(2)}\setminus S^\prime$ as in {\it (3)}. By Corollary~\ref{w4fine}, $\eta_s\in w_{f_1}-w_{f_2}+U^\bullet$ for some $f_1, f_2\in G^{(5)}$. Then, by {\it (2)} and Proposition~\ref{4f}, $\eta_s\in U$, whence, by Proposition~\ref{2xS4}.(2) and Corollary~~\ref{4f}, $V(S)\leq U$. The explicit formula for $\eta_s$ has been computed with GAP in~\cite[algebraA6(2B).g]{code}. 

Finally, let $H$ be a subgroup of $G$ isomorphic to $A_5$ as in {\it (4)} and let $t\in H^{(2)}$. By Proposition~\ref{2eta} and Corollary~\ref{4f}, $\gamma_t\in U^\bullet=U$, so the result follows by Lemma~\ref{A5}.   
\end{proof} 

Let now $t\in G^{(2)}$ and let $K_1$ and $K_2$ be the two subgroups of $G$ isomorphic to $S_4$ such that $t$ is contained in the derived subgroup of $K_i$, $i\in \{1,2\}$. Denote by $\gamma_{t,i}$ the element $\gamma_t$ defined as in Equation~(\ref{Eqgamma}) in the subalgebra $V(K_i)$.

\begin{cor}\label{corgamma}
With the above notation, let $t\in G^{(2)}$. Then, for $N\in  \{2,3,4\}$ and $i\in \{1,2\}$, there exist $[\gamma_{t,i}]_N\in V^{(NA)}$ such that
$\gamma_{t,i}=[\gamma_{t,i}]_2+[\gamma_{t,i}]_3+[\gamma_{t,i}]_4$.
\end{cor}
\begin{proof}
  This follows from Equations~(\ref{formulagamma})  and~(\ref{eqneta}).
\end{proof}

We conclude this section by determining the decomposition into irreducible submodules of $V^\circ$. 

For a partition $\lambda$ of $6$, denote by $V^{(\lambda)}$ the subspace of $V$ spanned by the $N$-axes in $V$ associated to a permutation of cycle type $\lambda$ and denote by $\Sp^{\lambda}$ the Specht module associated to $\lambda$.

\begin{theorem}\label{decv}
The structure of $\R[A_6]$-module of $V^\circ$ lifts naturally to a structure of $\R[S_6]$-module. With the above notation, the decomposition of $V^\circ$ into irreducible $\R[S_6]$-submodules is the following
$$V^\circ\cong 3\Sp^{(6)}\oplus 3\Sp^{(5,1)}\oplus 5\Sp^{(4,2)}\oplus 2\Sp^{(3,2,1)}\oplus \Sp^{(3^2)}\oplus 3\Sp^{(2^3)}\oplus \Sp^{(2,1^4)}\oplus \Sp^{(1^6)} .$$
\end{theorem}
\begin{proof}
By Equation~(\ref{defU}) and Corollary~\ref{coreta}, 
\begin{equation} \label{pp}
V^\circ= (V^{(2A)}+V^{(3A)}+V^{(4A)})\oplus \langle w\rangle.
\end{equation}
Let $$\zeta_w\colon S_6 \to GL(\langle w \rangle)$$ be the alternating representation and denote by $\langle w \rangle_\zeta$ the corresponding $\R[S_6]$-module, so that
\begin{equation}
    \label{wsigma}
  \langle w \rangle_\zeta\cong \Sp^{(1^6)}.
\end{equation}
Further, 
for $N\in \{2,3,4\}$, let 
$\mathfrak P_N$ be the $\R[S_6]$-permutation module induced by the action by conjugation of $S_6$ on the set $\mathcal C_N$ of the cyclic subgroups of order $N$ in $A_6$, and let  
$\mathfrak P$ be the direct sum of $\R[S_6]$-modules 
$\mathfrak P_2\oplus\mathfrak P_3\oplus\mathfrak P_4$. 
Since $U$ is a homomorphic image of $\mathfrak P$ (as $\R[A_6]$-modules),  the inner product on $V$ induces in an obvious way an inner product  
$$
\kappa \colon \mathfrak P\times \mathfrak P\to \R
$$ 
on $\mathfrak P$ (see e.g.~\cite[Section 2]{A12}). By Remark~\ref{rem2}, $\rad(\kappa)$ is a $\R[S_6]$-submodule of $\mathfrak P$ and  $U\cong\mathfrak P/\rad(\kappa)$, thus the permutation representation  
$$\zeta_{\mathfrak P}\colon S_6 \to GL(\mathfrak P)$$ 
defines a representation 
$$\zeta_U\colon S_6\to GL(U).$$ 
By construction,  $\zeta_U \oplus\zeta_w$ is a representation of $S_6$ on $U \oplus \langle w \rangle\cong  V^\circ $,
whose restriction to $A_6$ is induced by the identity map on $A_6$ as the Miyamoto group of $V$, proving the first assertion. 

Let  $N\in \{2,4\}$, by Proposition~\ref{dime}, $$
\dim(V^{(NA)})=45=|\mathcal C_N|=\dim(\mathfrak P_N),
$$ whence, by~\cite[Lemma~6]{FIM2}, 
\begin{equation}\label{mod2A}
    V^{(NA)}\cong \mathfrak P_N\cong \Sp^{(6)}\oplus \Sp^{(5,1)}\oplus 2\Sp^{(4,2)}\oplus  \Sp^{(3,2,1)}\oplus \Sp^{(2^3)} .
\end{equation}
Let now $N=3$. By Proposition~\ref{dime}, 
$$
\dim(V^{(3A)})=40=|\mathcal C_3|=\dim(\mathfrak P_3),
$$
and, since the exceptional automorphism of $A_6$ swaps $3$-cycles with cycles of type $3^2$, we have 
$$
V^{(3A)}=V^{(3)}\oplus V^{(3^2)}\:\mbox{ and }\: \dim(V^{(3)})=\dim(V^{(3^2)})=20.
$$ 
Let 
$$
V_+^{(3^2)}:=\langle u_{ab}+u_{ab^{-1}}\:|\:a,b \mbox{ disjoint } 3\mbox{-cycles in } G\rangle
$$
and 
$$
V_-^{(3^2)}:=\langle u_{ab}-u_{ab^{-1}}\:|\:a,b \mbox{ disjoint } 3\mbox{-cycles in } G\rangle .
$$
Clearly $V^{(3^2)}$ is the sum of $V_+^{(3^2)}$ and $V_-^{(3^2)}$. Since 
$$\dim (V_+^{(3^2)})=\dim(V_-^{(3^2)})=10=\tfrac{1}{2}\dim(V^{(3^2)}),
$$ we have
\begin{equation}\label{V32}
    V^{(3^2)}=V_+^{(3^2)}\oplus V_-^{(3^2)}.
\end{equation}
Since $V_+^{(3^2)}$ is isomorphic to the permutation module associated to the action of $S_6$ on the $10$ conjugates of $\langle (1,2,3), (4,5,6)\rangle$, an easy computation shows that 
\begin{equation}\label{10}
V_+^{(3^2)}\cong \Sp^{(6)}\oplus \Sp^{(4,2)} .
\end{equation} 
By~\cite[Theorem~1.1]{FIM3}, $V^{(3^2)}\cong \Sp^{(6)}\oplus \Sp^{(4,2)}\oplus \Sp^{(2^3)}\oplus \Sp^{(2,1^4)}$, whence by Equation~(\ref{V32}),
\begin{equation}
    \label{v32}
V_-^{(3^2)}\cong V^{(3^2)}/V_+^{(3^2)}\cong \Sp^{(2^3)}\oplus \Sp^{(2,1^4)}.
\end{equation}
By the Pasechnik's relations~\cite[Lemma~3.4]{A67},  $V_+^{(3^2)}\subseteq V^{(2A)}\oplus V^{(3)}$,   
whence
\begin{eqnarray*}
    \dim (U) &\leq &\dim(V^{(2A)})+\dim(V^{(3)})+\dim(V_-^{(3^2)})+\dim(V^{(4A)})\\
    &\leq&  45+20+10+45=120.
\end{eqnarray*}
By Corollary~\ref{coreta}, $\dim(U)=120$. It follows that 
\begin{eqnarray*}
    U&=&V^{(2A)}\oplus V^{(3)}\oplus V_-^{(3^2)} \oplus V^{(4A)}.
\end{eqnarray*}
Finally, since $V^{(3)}$ is isomorphic to the permutation module of $S_6$ on its subsets of order $3$, we have \begin{equation}
    \label{xxxx}
V^{(3)}\cong \Sp^{(6)}\oplus \Sp^{(5,1)}\oplus \Sp^{(4,2)}\oplus \Sp^{(3^2)}. 
\end{equation} 
The result follows by Equations~(\ref{mod2A}), (\ref{v32}), and (\ref{xxxx}).
\end{proof}


\section{The algebra closure of $V^\circ$}\label{closure}

In this section we prove that there is a unique possibility for the algebra product in $V$ and $V^\circ=V$. 

We first show that $a_t\cdot V^\circ \subseteq V^\circ$, for every $t\in G^{(2)}$.

\begin{lemma}\label{2-3-closure}
For every $t\in G^{(2)}$, $a_{t}\cdot (V^{(2A)}+V^{(3A)})\subseteq V^\circ$.
\end{lemma}
\begin{proof}
By the definition of $V^\circ$,  $a_{t}\cdot V^{(2A)}\subseteq V^\circ$. Since for every $h\in G^{(3)}$, $\langle t, h\rangle$ is contained in a maximal subgroup of $G$ isomorphic to $S_4$ or  to $A_5$ (see Table~\ref{tau3}), by Corollary~\ref{coreta},  we get $a_{t}\cdot V^{(3A)}\subseteq V^\circ$.   
\end{proof}

To deal now with $V^{(4A)}$ and $w$, we show we can find a set $\Lambda$ of vectors generating  $V^{(4A)}$ and   such that, for every $v\in \Lambda$,  the product $a_t\cdot v$ lies in $V^\circ$.

\begin{lemma}\label{some41}
Let $t\in G^{(2)}$ and $g\in G^{(4)}$ and suppose that $\langle t, g\rangle \cong S_4$. Then, with the notation of Equation~(\ref{ov}),  $a_{t}\cdot [\ov_{g}]_4 \in V^\circ$.
\end{lemma}
\begin{proof}
Since $\langle t, g\rangle \cong S_4$, by Corollary~\ref{coreta}, $V^\circ$ contains $a_{t}\cdot \ov_{g}$. Hence, by Lemma~\ref{L64} and Corollary~\ref{2-3-closure},  we have
$$
a_{t}\cdot [\ov_{g}]_4=a_{t}\cdot \ov_{g}-a_{t}\cdot([\ov_{g}]_2+[\ov_{g}]_3)\in V^\circ.
$$ 
\end{proof}

\begin{lemma}\label{some4}
Let $t\in G^{(2)}$ and $f\in G^{(5)}$ and suppose that $\langle t, f\rangle$ is properly contained in  $G$. Then, with the notation of Equation~(\ref{wf4}), $a_{t}\cdot [w_f]_4 \in V^\circ$.
\end{lemma}
\begin{proof}
Set
$$
E_t:=\{(f_1,f_2)\in G^{(5)}\times G^{(5)}\:|\:  \langle t, f_i \rangle < G,\;\forall  i \in \{1,2\}\}.
$$
and $$
F_t:=\{(f_1,f_2)\in G^{(5)}\times G^{(5)}\:|\: \langle f_1^t\rangle =\langle f_2\rangle \}.
$$ 
We first prove that
\begin{equation}
\label{f1f2}
\mbox{ for every } (f_1, f_2)\in E_t\cup F_t,\:\: a_{t}\cdot (w_{f_1}-w_{f_2})\in V^\circ.
\end{equation}
Assume $(f_1,f_2)\in E_t$ and let $H_i$ be a maximal subgroup of $G$ containing  $\langle t, f_i \rangle$. Then $H_i \cong A_5$ and 
$a_{t}$ and $w_{f_i}$ are contained in the subalgebra $V( H_i)$. Hence, by Lemma~\ref{A5}, there is a unique possibility for the product $a_{t}\cdot w_{f_i}$, and, by Corollary~\ref{coreta}.(4),  
this product belongs to $V^\circ$. 
Next assume $(f_1,f_2)\in F_t$. 
Since $\langle f_1^t\rangle =\langle f_2\rangle $, by the choice of the set $G^{(5)}$, $f_1$ and $f_2$ are conjugate in $G$. Hence, either $f_1^t=f_2$, or $f_1^t=f_2^{-1}$. In both cases $(w_{f_1})^t=w_{f_1^t}=w_{f_2}$. Thus $w_{f_1}-w_{f_2}$ is a $\tfrac{1}{32}$-eigenvector for $\ad_{a_{t}}$ and so  $a_{t}\cdot (w_{f_1}-w_{f_2})=\tfrac{1}{32}(w_{f_1}-w_{f_2})\in V^\circ$, proving~(\ref{f1f2}). 

Set
\begin{eqnarray*}
W_t&:=&\langle [w_{f_1}]_4-[w_{f_2}]_4 \:\mid \: (f_1,f_2)\in E_t\cup F_t\rangle,\\
&\\
W_t^{\prime}
&:=&\langle v_g \:\mid \: \langle t, g\rangle \cong D_8\mbox{ or } \langle t, g\rangle \cong  S_4\rangle,\\
&\\
W_t^{\prime\prime}
&:=&\langle v_g-v_{g^t} \:|\: g\in G^{(4)}\rangle, \mbox{ and }\\
&\\
\overline W_t&:=&W_t+W_t^{\prime}+W_t^{\prime\prime}.\\
\end{eqnarray*}
A direct calculation using GAP (see~\cite[algebraA6(2B).g]{code}) shows that, for $f\in G^{(5)}$  such that   
$\langle t, f \rangle < G$, we have
$
\dim(\langle [w_{f}]_4, \overline W_t\rangle)=\dim(\overline W_t)=36,
$
whence 
\begin{equation}
    \label{xud}
[w_{f}]_4\in  \overline W_t.
\end{equation}
By Equation~(\ref{w}), for every $(f_1, f_2)\in G^{(5)}\times G^{(5)}$, we have 
$$a_{t}\cdot ([w_{f_1}]_4-[w_{f_2}]_4)=a_{t}\cdot (w_{f_1}-w_{f_2})-a_{t}\cdot ([w_{f_1}]_2+[w_{f_1}]_3-[w_{f_2}]_2-[w_{f_2}]_3),$$
whence,  by Equation~(\ref{f1f2}) and Corollary~\ref{2-3-closure}, we get 
\begin{equation}\label{f1f2bis}
a_{t}\cdot W_t\subseteq V^\circ.
\end{equation}
By claim {\it (3)} of Corollary~\ref{coreta}, we have also 
\begin{equation}\label{f1f2bisbis}
a_{t}\cdot W_t^{\prime}\subseteq V^\circ,
\end{equation}
and finally 
\begin{equation}\label{f1f2bisbisbis}
a_{t}\cdot W_t^{\prime\prime}\subseteq V^\circ,
\end{equation}
since $v_g-v_{g^t}$ is a $\tfrac{1}{32}$-eigenvector for $\ad_{a_{t}}$. Thus, by Equations~(\ref{xud}), (\ref{f1f2bis}), (\ref{f1f2bisbis}), and~(\ref{f1f2bisbisbis}),    $a_{t}\cdot [w_f]_4\in a_{t}\cdot {\overline W}_t\subseteq V^\circ$.
\end{proof}

\begin{lemma}\label{prodw}
Let $w$ be as in Equation~(\ref{somma5}). For every $t\in G^{(2)}$, $a_{t}\cdot w\in V^\circ$.
\end{lemma}
\begin{proof}
Let $f\in G^{(5)}$ be such that $f^t=f^{-1}$. Then, $w_f$ is contained in the subalgebra $\langle \langle a_{t}, a_{tf}\rangle \rangle$, which is of type $5A$.  By Table~\ref{NS}, $a_{t}\cdot w_f\in V^\circ$ . Hence, by Equation~(\ref{w}), Corollary~\ref{2-3-closure} and Lemma~\ref{some4}, we have
$$
a_{t}\cdot w=36a_{t}\cdot (w_f-[w_f]_2-[w_f]_3-[w_f]_4)\in V^\circ .
$$
\end{proof}

\begin{lemma}\label{products2-4}
For every $t\in G^{(2)}$, $a_{t}\cdot V^{(4A)}\subseteq V^\circ$.
\end{lemma}
\begin{proof}
Let $t\in G^{(2)}$ and let $K_1$ and $K_2$ be the two subgroups of $G$ isomorphic to $S_4$ such that $t$ is contained in the derived subgroup of $K_i$, $i\in \{1,2\}$. Let $\gamma_{t,i}$ and $[\gamma_{t,i}]_N$, for $N\in \{2,3,4\}$, be as in Corollary~\ref{corgamma}. Moreover, for $i\in \{1,2\}$, let $L_i$ and $L_i^\ast$ be the two subgroups of $G$ isomorphic to $A_5$ containing the derived subgroup of $K_i$.
Let
\begin{equation}\label{}
Q_t:=\langle [\gamma_{t,i}]_4^x \:|\:  x\in L_i\cup L_i^\ast, i\in \{1,2\}\rangle .
\end{equation}
For each $i\in \{1, 2\}$, the vectors $a_{t}$ and $\gamma_{t,i}$ belong to $V(L_i)\cap V( L_i^\ast)$. Thus, for each $x\in L_i\cup L_i^\ast$, we have 
 $(\gamma_{t,i})^x\in V(L_i)\cup V(L_i^\ast)$ and so, by Corollary~\ref{coreta},  
 $a_{t}\cdot (\gamma_{t,i})^x\in V(L_i)\cup V(L_i^\ast)\subseteq V^\circ$.  Thus, by Corollaries~\ref{corgamma} and~\ref{2-3-closure}, 
\begin{equation}\label{q}
  a_{t}\cdot Q_t\subseteq  V^\circ.  
\end{equation}
Let $W_t^\prime$ and  $W_t^{\prime\prime}$ be the subspaces defined in the proof of Lemma~\ref{some4} and set 
$$
R_t:=\langle [w_f]_4\:|\: f\in G^{(5)}, \langle t,f\rangle \le G\rangle\:
\mbox{
and }\:
\overline R_t:=R_t+W_t^\prime+W_t^{\prime\prime}.
$$
As in the proof of Lemma~\ref{some4}, by Equation~(\ref{q}), Lemma~\ref{some4} and  Corollary~\ref{coreta}, $a_{t}\cdot (Q_t+\overline R_t)\subseteq V^\circ$.
A direct computation in GAP (see~\cite[algebraA6(2B).g]{code}) shows that 
$$
\dim(Q_t+\overline R_t)=45.
$$
Since $Q_t+\overline R_t \leq V^{(4A)}$ and, by Proposition~\ref{dime}, $\dim(V^{(4A)})=45$, 
it follows that $Q_t+\overline R_t=V^{(4A)}$ and the result follows.
\end{proof}

\begin{proposition}\label{aV}
For every $t\in G^{(2)}$, $a_{t}\cdot V^\circ \subseteq V^\circ$.
\end{proposition}
\begin{proof}
The result follows immediately by Corollary~\ref{coreta}.(1), Corollary~\ref{2-3-closure},  Lemma~\ref{prodw}, and Lemma~\ref{products2-4}.
\end{proof}

We can now use the resurrection principle in its simpler version of~\cite[Lemma~1.8]{IPSS} to get the remaining products between two odd axes.

\begin{lemma}\label{oddprod1}
Let $h,k\in G$ such that $x_h$ and $x_k$ are odd axes of $V$.  Suppose there is $t\in G^{(2)}$ inverting by conjugation both $h$ and $k$. Then $x_h\cdot x_k\in V^\circ$. 
\end{lemma}
\begin{proof}
 $\langle \langle a_{t}, x_h\rangle\rangle$ is contained in a Norton-Sakuma algebra of type $|h|A$ and similarly $\langle \langle a_{t},x_k\rangle\rangle$ is contained in a Norton-Sakuma algebra of type $|k|A$. Let $e_h$ be a $0$-eigenvector for $\ad_{a_{t}}$ and let $e_k$ and $\tilde e_k$ be a $0$-eigenvector and a $ \tfrac{1}{4}$-eigenvector  for $\ad_{a_{t}}$, respectively. Using~\cite[Table 4]{IPSS}, we can express $e_h$ (resp. $e_k$ and $\tilde e_k$) as a linear combination of axes and $x_h$ (resp. $x_k$ and $\tilde x_k$).  It follows that there  exist $y$ and $z$ in $V^\circ$, such that
$$
e_h\cdot e_k=x_h\cdot x_k + y\: \mbox{ and } \:e_h\cdot \tilde e_k=x_h\cdot x_k + z.
$$
By the fusion law, $e_h\cdot e_k$ is a $0$-eigenvector for $\ad_{a_{t}}$, while $e_h\cdot e_k$ is a $\tfrac{1}{4}$-eigenvector for $\ad_{a_{t}}$.
By~\cite[Lemma~1.8]{IPSS} and Proposition~\ref{aV}, we have 
$$
x_h\cdot x_k =-[4a_{t}\cdot (y-z)+z]\in V^\circ .
$$
\end{proof}

\begin{cor}\label{V3Vc}
    $V^{(3)}\cdot V^\circ\subseteq V^\circ$
\end{cor}
\begin{proof}
 The result follows by Lemma~\ref{oddprod1} since, for every $h,k\in G^{(3)}$, $g\in G^{(4)}$, and $f\in G^{(5)}$ there exists $t\in G^{(2)}$ inverting by conjugation both $h$ and $k$, or $h$ and $g$, or $h$ and $f$.   
\end{proof}

\begin{proposition}\label{oddprod2}
For every 
$g,l\in G^{(4)}$ and $f_1, f_2\in G^{(5)}$ the following products  belong to $V^\circ$:
\begin{enumerate}
\item $v_g\cdot v_g$ and $w_{f_1}\cdot w_{f_1}$;
\item $v_g\cdot v_l$,  if $\langle g, l\rangle \leq S_4$;
\item $v_g\cdot v_l$,  if $\langle g, l\rangle\cong A_6$ and $l\in G^{(4)}_5(g)\cap G^{(4)}_5(g^{-1})$;
\item $v_g\cdot w_{f_1}$,  if $\{f_1, f_1^2\}\cap G^{(5)}_5(g)\cap G^{(5)}_3(g^{-1})\neq \emptyset$;
\item $w_{f_1}\cdot w_{f_2}$,  if $f_2\not \in G^{(5)}_5(f_1)\cap G^{(5)}_5(f_1^{-1})$;
\item $v_g\cdot w_{f_1}$,  if $f_1\in G^{(5)}_5(g)\cap G^{(5)}_5(g^{-1})$.
\end{enumerate}
\end{proposition}
\begin{proof}
The first assertion follows by the Norton-Sakuma Theorem and Table~\ref{NS}.  The second claim follows by Corollary~\ref{coreta}. The third, fourth and the fith claims follow by Lemma~\ref{oddprod1} since, for every $g,l\in G^{(4)}$, and $f_1, f_2\in G^{(5)}$ as in the statement, there exists $t\in G^{(2)}$ inverting by conjugation both $g$ and $l$, or $g$ and $f_1$, or $f_1$ and $f_2$. 
Finally, suppose  $f\in G^{(5)}_5(g)\cap G^{(5)}_5(g^{-1})$. Then, as in the proof of Lemma~\ref{45},  we may assume $g=(1,2)(3,4,5,6)$ and $f=(2,3,4,6,5)$. Set $t:=g^2=(3,5)(4,6)$. 
 By Table~\ref{tau4}, it follows that $(a_{t},v_g)=0$, whence, by Lemma~\ref{idem}, $v_g$ is a $0$-eigenvector for $\ad_{a_{t}}$. Moreover, since $t$ inverts $f$, the algebra $\langle \langle a_{t}, a_{tf}\rangle \rangle$ is a Norton-Sakuma algebra of  type $5A$, whence, by Table~\ref{NS},  $$
w_f=a_{t}\cdot a_{tf}-\tfrac{1}{2^7}(3a_{t}+3a_{tf}-a_{tf^2}-a_{tf^3}-a_{tf^4}).
$$ 
By~\cite[Lemma~1.10]{IPSS} and Proposition~\ref{aV}, it follows that 
\begin{align*}
v_g\cdot w_f=& v_g\cdot (a_{t}\cdot a_{tf}-\tfrac{1}{2^7}(3a_{t}+3a_{tf}-a_{tf^2}-a_{tf^3}-a_{tf^4}))\\
=& v_g\cdot (a_{t}\cdot a_{tf})-\tfrac{1}{2^7}v_g\cdot(3a_{t}+3a_{tf}-a_{tf^2}-a_{tf^3}-a_{tf^4})\\
=& (v_g\cdot a_{tf})\cdot a_{t}-\tfrac{1}{2^7}v_g\cdot(3a_{t}+3a_{tf}-a_{tf^2}-a_{tf^3}-a_{tf^4})\in V^\circ.
\end{align*}
\end{proof}

\begin{proposition}\label{44final}
$V^{(4A)}\cdot V^{(4A)}\subseteq V^\circ$.
\end{proposition}
\begin{proof}
Let $g\in G^{(4)}$. Define
\begin{align*}
F_g:=&\{f\in G^{(5)}\:|\: \mbox{ either }\{f, f^2\}\cap G^{(5)}_5(g)\cap G^{(5)}_3(g^{-1})\neq \emptyset \\
&\mbox{ or } f\in G^{(5)}_5(g)\cap G^{(5)}_5(g^{-1}) \}.
\end{align*}
By Proposition~\ref{oddprod2}, for every $f\in F_g$, we have $v_g\cdot w_f\in V^\circ$, whence, by Equation~(\ref{w}), Proposition~\ref{aV} and Proposition~\ref{oddprod2},  $v_g\cdot ([w_{f_1}]_4-[w_{f_2}]_4)\in V^\circ$  for every $ f_1,f_2\in F_g$.  Moreover, for every $x\in N_{G}(\langle g\rangle)$, $(v_g)^x=v_g$, whence
\begin{equation}\label{eqw4bis}
 v_g\cdot ([w_{f_1}]_4-[w_{f_2}]_4)^x\in V^\circ \:\mbox{ for every } \: f_1,f_2\in F_g \:\mbox{ and }\: x\in N_{ G}(\langle g\rangle).  
\end{equation}
Let $K_1$ and $K_2$ be the two subgroups of $G$ isomorphic to $S_4$ containing $g$. For every involution $s\in K_i\setminus K_i^\prime$, denote by $\eta_{s,i}$ the element in the subalgebra $V(K_i)$ defined as in Equation~(\ref{Eqeta}). By Corollary~\ref{coreta} (Equation~(\ref{eqneta})),  for $N\in \{2,3,4\}$, there exist $[\eta_{s,i}]_N\in V^{(NA)}$ such that 
$$\eta_{s,i}=[\eta_{s,i}]_2+[\eta_{s,i}]_3+[\eta_{s,i}]_4.
$$
By Corollary~\ref{coreta}, $v_g\cdot \eta_{s,i}\in V( K_i)\subseteq V^\circ$, whence, by Corollary~\ref{2-3-closure} and Proposition~\ref{oddprod2}, 
\begin{equation}\label{vg4}
v_g\cdot [\eta_{s,i}]_4\in V^\circ \:\mbox{ for every }\: s\in K_i\setminus K_i^\prime,\: i\in \{1,2\}.    
\end{equation}
Now define
\begin{eqnarray*}
Q_g^1&:=&\langle v_l\:|\: \langle g, l\rangle \leq S_4 \rangle\\
Q_g^{2}&:=&\langle v_l\:|\: l\in G^{(4)}_5(g)\cap G^{(4)}_5(g^{-1})\rangle \\
Q_g^{3}&:=& \langle ([w_{f_1}]_4-[w_{f_2}]_4)^x \:|\:  f_1,f_2\in F_g, x\in N_G(\langle g\rangle)\rangle\\
Q_g^{4}&:=&\langle [\eta_{s,i}]_4 \:|\: s\in K_i\setminus K_i^\prime, i\in \{1,2\}\rangle \\
Q_g&:=&Q_g^1+Q_g^{2}+Q_g^{3}+Q_g^{4}.
\end{eqnarray*}
A direct check (see~\cite[algebraA6(2B).g]{code}) and Proposition~\ref{dime} show that $\dim(Q_g)=45=\dim(V^{(4A)})$, whence 
$$Q_g=V^{(4A)}.
$$ 
By  Proposition~\ref{oddprod2}, Equation~(\ref{eqw4bis}), and Equation~(\ref{vg4}), $v_g\cdot Q_g^i\subseteq V^\circ$ for every $i\in \{1,2,3,4\}$, whence $v_g\cdot V^{(4A)}=v_g\cdot Q_g \subseteq V^\circ$.
\end{proof}

\begin{lemma}\label{prodfinal}
The following hold
\begin{enumerate}
\item $V^{(4A)}\cdot V^{(5A)}\subseteq V^\circ$;
\item $w\cdot w\in V^\circ$.
\end{enumerate}
\end{lemma}
\begin{proof}
Let $g\in G^{(4)}$ and $\bar f\in G^{(5)}_5(g)\cap G^{(5)}_3(g^{-1})$. 
By Lemma~\ref{rel5assi}, Proposition~\ref{aV}, Corollary~\ref{V3Vc}, Proposition~\ref{oddprod2}, and 
Proposition~\ref{44final} we have
\begin{equation}\label{ultima}
v_g\cdot w= v_g\cdot 36(w_{\bar f}-[w_{\bar f}]_2-[w_{\bar f}]_3-[w_{\bar f}]_4)\in V^\circ . 
\end{equation}
It follows that, for every $f\in G^{(5)}$, 
$$
v_g\cdot w_f=v_g\cdot (\tfrac{1}{36}w+[w_f]_2+[w_f]_3+[w_f]_4)\in V^\circ,
$$
proving {\it (1)}.
The second assertion follows again by Lemma~\ref{rel5assi}, Proposition~\ref{aV}, Corollary~\ref{V3Vc},
Proposition~\ref{44final} and claim {\it (1)} since 
$$
w\cdot w= 36^2w_f\cdot w_f-36^2([w_f]_2+[w_f]_3+[w_f]_4)\cdot (w_f-[w_f]_2-[w_f]_3-[w_f]_4)
$$
and $w_f\cdot w_f\in V^\circ$ by the Norton-Sakuma Theorem.
\end{proof}

\begin{proposition}
$V=V^\circ$ and the algebra product in $V$ is unique.
\end{proposition}
\begin{proof}
By Corollary~\ref{coreta}(1), Proposition~\ref{aV}, Corollary~\ref{V3Vc}, Proposition~\ref{44final}, and Lemma~\ref{prodfinal}, $V^\circ$ is closed under the algebra product. Since it contains the generating axes $a_{t}$, $t\in T$, we have $V^\circ=V$. Moreover, the arguments used in this section show that every product is uniquely determined by the two factors.  
\end{proof}

\section{Acknowledgements}
This work was supported by the ``National Group for Algebraic and Geometric Structures, and their Applications'' (GNSAGA-INdAM)
.


\end{document}